\documentclass[11pt,a4paper,oneside]{article}
\usepackage{graphicx}
\usepackage{amsmath,amsthm,amssymb,url}
\usepackage{multirow,bigdelim}
\usepackage{kbordermatrix}

\newcommand{\ignore}[1]{}
\normalbaselines
\usepackage{amsmath}
\usepackage{amssymb}
\usepackage{mathtools}
\usepackage{mathcomp}
\usepackage{textcomp}
\usepackage{amsthm}
\usepackage{enumerate}
\usepackage[auth-lg]{authblk}
\usepackage{cite}
\usepackage{tikz}
\usepackage{caption}
\usepackage{subcaption}
\usepackage{bm}
\usetikzlibrary{arrows}
\usepackage{multicol}
\usepackage{tikz-qtree}
\usepackage{rotating}
\usepackage{url}
\usepackage{array}

\newtheorem{thm}{Theorem}[section]\theoremstyle{plain}
\newtheorem{theorem}[thm]{Theorem}\theoremstyle{plain}
\theoremstyle{plain}
\newtheorem{lemma}[thm]{Lemma}\theoremstyle{plain}
\theoremstyle{plain}

\theoremstyle{plain}
\newtheorem{claim}[thm]{Claim}\theoremstyle{plain}
\theoremstyle{plain}
\theoremstyle{plain}
\theoremstyle{plain}
\theoremstyle{plain}
\theoremstyle{plain}
\theoremstyle{plain}
\newtheorem{conjecture}[thm]{Conjecture}\theoremstyle{plain}
\theoremstyle{plain}

\DeclareMathOperator{\rank}{rank}

\DeclareMathOperator{\cl}{cl}

\newcommand{\val}{{\rm val}}

\newcommand{\cM}{{\mathcal M}}
\newcommand{\cN}{{\mathcal N}}
\newcommand{\cU}{{\mathcal U}}
\newcommand{\cC}{{\mathcal C}}
\newcommand{\cR}{{\mathcal R}}
\newcommand{\cS}{{\mathcal S}}

\newcommand{\cI}{{\mathcal I}}
\newcommand{\cH}{{\mathcal H}}
\newcommand{\cX}{{\mathcal X}}

\newcommand{\R}{{\mathbb R}}

\newcommand{\zed}{{\mathbb Z}}

\newcommand{\sm}{{\setminus}}
\usepackage{xcolor}
\usepackage{graphicx} 
\usepackage{booktabs} 
\usepackage{pgf,tikz}
\usetikzlibrary{positioning}
\usepackage{comment}
\ignore{
\usetikzlibrary{arrows,matrix,positioning, fit, backgrounds, calc,shapes.geometric}
\tikzstyle{mybox} = [draw=red, very thick,
    rectangle, rounded corners, inner sep=10pt, inner ysep=20pt]
\tikzstyle{fancytitle} =[fill=red, text=white]
\tikzset{
    invisible/.style={opacity=0,text opacity=0},
    visible on/.style={alt={#1{}{invisible}}},
    alt/.code args={<#1>#2#3}{
      \alt<#1>{\pgfkeysalso{#2}}{\pgfkeysalso{#3}}
    },
  }
\definecolor{ForestGreen}{rgb}{0.13, 0.55, 0.13}
\definecolor{FluorescentOrange}{rgb}{1.0, 0.75, 0.0}
\tikzset{%
  highlight/.style={rectangle,rounded corners,fill=green!15,draw,fill opacity=0.5,thick,inner sep=0pt}
}

\usetikzlibrary{patterns}
}

\begin{document}
\title{Maximal Matroids in Weak Order Posets}
\author{ 
Bill Jackson\thanks{School
of Mathematical Sciences, Queen Mary University of London,
Mile End Road, London E1 4NS, England. email:
{\tt B.Jackson@qmul.ac.uk}}  
 and Shin-ichi Tanigawa\thanks{Department of Mathematical Informatics, Graduate School of Information Science and Technology, University of Tokyo, 7-3-1 Hongo, Bunkyo-ku, 113-8656,  Tokyo Japan. email: {\tt tanigawa@mist.i.u-tokyo.ac.jp}}}

\date{\today}
\maketitle

\begin{abstract}
Let $\cX$ be a family of subsets of a finite set $E$. 
A matroid on $E$ is called an $\cX$-matroid if each set in $\cX$ is a circuit.
We consider the problem of determining when there exists a unique maximal $\cX$-matroid 
in the weak order poset of all $\cX$-matroids on $E$,
and characterizing its rank function when it exists.

\medskip

%
\end{abstract}
\section{Introduction}
\subsection{Unique maximality problem and submodularity conjecture}
Let 
$\cX$ be a family of subsets of a finite set $E$. 
We will refer to any matroid on $E$ in which each set in $\cX$ is a circuit as an 
{\em $\cX$-matroid on $E$}. 
The set of all $\cX$-matroids on $E$ forms a poset under the {\em weak order of matroids} in which, for two matroids $\cM_1$ and $\cM_2$ with the same groundset, we have $\cM_1\preceq \cM_2$ if every independent set in $\cM_1$ is independent in $\cM_2$. 
The main problem addressed in this paper is to determine when this poset has a unique maximal element and to characterise this unique maximal matroid when it exists.

Our key tool  is the following upper bound on the rank function of any  $\cX$-matroid on $E$ from~\cite{CJT2}. 
A {\em proper $\cX$-sequence} is a sequence $\cS=(X_1, X_2, \dots, X_k)$ of sets in $\cX$ such that $X_i\not\subset \bigcup_{j=1}^{i-1} X_j$ for all $i=2,\dots, k$. For $F\subset E$, let ${\rm val}(F,\cS)=| F\cup (\bigcup_{i=1}^k X_i)|-k$.

\begin{lemma}[{\cite[Lemma 3.3]{CJT2}}]\label{lem:upper1}
Suppose $\cM$ is an $\cX$-matroid on $E$ and $F\subseteq E$. Then $r_{\cM}(F)\leq {\rm val}(F,\cS)$ for any proper $\cX$-sequence $\cS$.  Furthermore, if equality holds, then $r_\cM(F-e)=r_\cM(F)-1$ for all 
$e\in F\sm  (\bigcup_{X\in \cS} X)$ and $r_\cM(F+e)=r_\cM(F)$ for all 
$e\in \bigcup_{X\in \cS} X$. 
\end{lemma}

We can use this lemma to derive a sufficient condition for the poset of all $\cX$-matroids on $E$ to have a unique maximal element. We need to consider a slightly larger poset. We say that a matroid $\cM$ on $E$ is {\em $\cX$-cyclic} if each $X\in \cX$ is a {\em cyclic set} in $\cM$ i.e. for every $e\in X$, there is a circuit $C$ of $\cM$ with $e\in C\subseteq X$.


\begin{lemma}\label{lem:upper2}
Let $\cX$ be a family of subsets of a finite set $E$ and define
${\rm val}_{\cX}:2^E\to \zed$ by
\begin{equation}\label{eq:f}
{\rm val}_{\cX}(F)=\mbox{$\min \{{\rm val}(F,\cS): \text{$\cS$ is a proper $\cX$-sequence}\}$} \qquad (F\subseteq E). 
\end{equation}
Suppose ${\rm val}_{\cX}$ is a submodular set function on $E$. Then ${\rm val}_{\cX}$ is the rank function of 
an $\cX$-cyclic 
matroid $\cM_{\cX}$ on $E$. 
In addition, if the poset of all $\cX$-matroids on $E$ is nonempty, then $\cM_{\cX}$ is the unique maximal  $\cX$-matroid on $E$.
\end{lemma}
\begin{proof}
It is straightforward to check that ${\rm val}_{\cX}$ is non-decreasing and satisfies ${\rm val}_{\cX}(e)\leq 1$ for all $e\in E$. Since ${\rm val}_{\cX}$ is also submodular, this implies that ${\rm val}_{\cX}$ is the rank function of a matroid $\cM_{\cX}$. To see that $\cM_{\cX}$ is $\cX$-cyclic, we choose $e\in X\in \cX$ and let $\cS$ be a proper $\cX$-sequence such that ${\rm val}_{\cX}(X-e)={\rm val}(X-e,\cS)$. If $e\in \bigcup_{X_i\in \cS}X_i$ then we have ${\rm val}_{\cX}(X)\leq {\rm val}(X,\cS)={\rm val}(X-e,\cS)={\rm val}_{\cX}(X-e)$. On the other hand, if $e\not\in \bigcup_{X_i\in \cS}X_i$ then
we can extend $\cS$ to a longer proper $\cX$-sequence $\cS'$ by adding $X$ as the last element of $\cS'$ and we will have   
${\rm val}_{\cX}(X)\leq {\rm val}(X,\cS')={\rm val}(X-e,\cS)={\rm val}_{\cX}(X-e)$. In both cases  equality must hold throughout since ${\rm val}_{\cX}$ is non-decreasing. Since ${\rm val}_{\cX}$ is the rank function of $\cM_{\cX}$, the equality ${\rm val}_{\cX}(X)={\rm val}_{\cX}(X-e)$ implies that $e$ belongs to a circuit of $\cM_{\cX}$ which is contained in $X$. Hence $\cM_{\cX}$ is $\cX$-cyclic.

Lemma \ref{lem:upper1} implies that $\cM\preceq \cM_{\cX}$ for every $\cX$-matroid $\cM$ on $E$.  If there exists at least one $\cX$-matroid on $E$, then this implies that each $X\in \cX$ is a circuit in $\cM_{\cX}$ and that $\cM_{\cX}$ is the unique maximal  $\cX$-matroid on $E$.
 \end{proof}

We conjecture that the converse to Lemma \ref{lem:upper2} is also  true.
The special case when $\cX$ is the set of all non-spanning circuits of a matroid on $E$ was previously given in \cite{CJT2}. 
\begin{conjecture}\label{con:unique}
Let $\cX$ be a family of subsets of a finite set $E$.
Suppose there is at least one $\cX$-matroid on $E$.
Then the poset of all  $\cX$-matroids on $E$ has a unique maximal element if and only if ${\rm val}_{\cX}$ is a submodular set function on $E$.~\footnote{More generally, if we remove the hypothesis that there is at least one $\cX$-matroid on $E$, then we conjecture that the poset of all  $\cX$-cyclic matroids on $E$  has a unique maximal element if and only if ${\rm val}_{\cX}$ is a submodular set function on $E$.
}
\end{conjecture}

\noindent
We will verify this conjecture
for various families $\cX$ and provide some tools to facilitate  further progress on the conjecture.

Conjecture~\ref{con:unique} is motivated by the polynomial identity testing problem of symbolic determinants (or the Edmonds problem). In this problem, we are given a matrix $A$ with entries in $\mathbb{Q}[x_1,\dots, x_n]$, and we are asked to decide whether the rank of $A$ over $\mathbb{Q}(x_1,\dots, x_n)$ is at least a given number. The Schwarz-Zippel Lemma implies that the problem is in the class NP, but it is a long-standing open problem to show that it is also in co-NP. 
The following experimental approach may aid our understanding of this problem. We  first test the linear independence/dependence of small sets of rows of $A$ 
to obtain a family $\cX$ of minimally dependent sets of rows.
Then Lemma \ref{lem:upper1} tells us that we can use any $\cX$-sequence to obtain a certificate that the rank of $A$ is at most a specified value.
In addition, if the "freest" matroid on the 
groundset $E$ indexed by the rows of $A$ in which each set in  $\cX$ is a circuit is 
uniquely determined, then Conjecture~\ref{con:unique} would imply that
its rank is ${\rm val}_{\cX}$ and this function has the potential to be the rank function of the row matroid of $A$. 

\subsection{Unique maximality problem on graphs}
We will concentrate on the special case of Conjecture \ref{con:unique} when $E$ is the edge set of a graph $G$ and $\cX$ is the family $\cH_G$ of  edge sets of all subgraphs of $G$ which are isomorphic to some member of a given family $\cH$ of graphs.
To simplify terminology we say that a matroid $\cM$ is a {\em $\cH$-matroid on $G$} if 
it is an $\cH_G$-matroid on $E(G)$. 
We will assume throughout that $G$ contains at least one copy of each $H\in \cH$ otherwise we can just consider $\cH\sm \{H\}$. 
This implies that the edge sets of any two subgraphs of $G$ which are  isomorphic to the same subgraph of a graph $H\in \cH$ will have the same rank in $\cM$, but we do not require $\cM$ to be  completely {\em symmetric} i.e.   the edge sets of every pair of isomorphic subgraphs of $G$  have the same rank.

We will simplify notation in the case 
when $\cH=\{H\}$ and refer to a $\cH$-matroid on $G$ as a {\em $H$-matroid on $G$}.
Two examples of $K_3$-matroids on $K_n$ are the graphic matroid of $K_n$ and the rank two uniform matroid on $E(K_n)$.

Chen, Sitharam and Vince previously considered the unique maximality problem for $H$-matroids on $K_n$ for various graphs $H$. They announced at a workshop at BIRS in 2015, see \cite{Stalk}, that there is a unique maximal $K_5$-matroid on $K_n$. Sitharam and Vince subsequently released a preprint \cite{SV} which claims to show that  there is a unique maximal $H$-matroid on $K_n$ for {\em all} graphs $H$. Unfortunately their claim is false. Pap~\cite{Pap} pointed out that the poset of $C_5$-matroids on $K_n$ has two maximal elements. We will describe Pap's counterexample, and give other counterexamples to the Sitharam-Vince claim in Section \ref{sec:notunique}.

Our interest in this topic was motivated by the work of 
Graver, Servatius, and Servatius~\cite{G,GSS} and Whiteley~\cite{Wsurvey} on maximal abstract rigidity matroids, and that of Chan, Sitharam and Vince~\cite{Stalk,SV} on maximal $H$-matroids. 
In two joint papers with Clinch \cite{CJT1,CJT2}, we were able to   confirm 
that there is a unique maximal $K_5$-matroid on $K_n$ and, more importantly, give a good characterisation for the rank function of this matroid. The theory of matroid erections due to Crapo \cite {C} is a key ingredient in our proof technique.

In this paper we will use results on matroid erection from \cite{CJT2} to construct a 
maximal element in the poset of all $\cX$-matroids on a set  $E$. 
  We will show that this element is the unique maximal element in the poset of all $\cH$-matroids on a graph $G$ for various pairs $(\cH,G)$, and verify that Conjecture~\ref{con:unique} holds in each case.

\subsection{Weakly saturated sequences}
The function ${\rm val}_{\cX}$ defined in (\ref{eq:f}) is related to the {\em weak saturation number} in extremal graph theory.
Let  $\cX$ be a family of subsets of a finite set $E$, and $F_0\subseteq E$.
A proper $\cX$-sequence $(X_1, X_2,  \dots, X_m)$ is said to be 
a {\em weakly $\cX$-saturated sequence from $F_0$} if 
$|X_i\setminus (F_0\cup \bigcup_{j<i} X_j)|=1$ for all $i$ with $1\leq i\leq m$. We say that $E$ {\em can be constructed  by a weakly $\cX$-saturated sequence from $F_0$} if there is a weakly $\cX$-saturated sequence $\cS$ from $F_0$ with $E=F_0\cup \bigcup_{X\in \cX} X$.
These sequences were first introduced by Bollob\'as \cite{B}, where 
he posed the problem of determining the size of a smallest set $F_0$ from which  $E$ can be constructed by a weakly $\cX$-saturated sequence. 
The problem has subsequently been studied by several authors, typically in the case when $E$ is the edge set of a complete $k$-uniform hypergraph or a complete bipartite graph, see for example \cite{A,B,Khyp,Ksym,MS,P1,P2}.
We will see in Sections 3 and 4 that results on weakly $\cX$-saturated sequences can sometimes be used to prove the unique maximality of an $\cX$-matroid. However this approach is applicable only when the flats of the target matroid are easily described. (The difficulty of deciding uniqueness when the structure of the flats is more complicated is illustrated by the matroids discussed in Section~6.)

The concept of $\cX$-matroids was previously studied by Kalai~\cite{Khyp} and  Pikhurko~\cite{P2} with the  goal of constructing a maximum rank $\cX$-matroid on $E$ to obtain 
a lower bound on the size of a  set $F_0$ from which  $E$ can be constructed by a weakly $\cX$-saturated sequence.
Our concern in this paper is different: we would like to gain a better understanding of the poset of all $\cX$-matroids on a given finite set $E$ by determining its maximal elements.
%

\medskip

We close this section by listing notation used throughout the paper.
Let $\cM$ be a matroid on a finite set.
Its rank function and closure operator are denoted by $r_{\cM}$ and ${\rm cl}_{\cM}$, respectively.
A set $F\subseteq E$ with ${\rm cl}_{\cM}(F)=F$ is called a {\em flat}.

For a graph $G$, $V(G)$ and $E(G)$ denote its vertex set and its edge set, respectively. Let $N_G(v)$ be the set of neighbors of $v$ in $G$. For $F\subseteq E(G)$, let $V(F)$ be the set of vertices incident to $F$ and let $G[F]$ be the graph with vertex set $V(F)$ and edge set $F$. Let $d_F(v)$ be the number of edges in $F$ incident to a vertex $v\in V(G)$, and let $N_F(v)$ be the set of neighbors of $v$ in $G[F]$.

For disjoint sets $X$ and $Y$, let $K(X)$ be the complete graph with vertex set $X$
and $K(X;Y)$ be the complete bipartite graph with vertex partition $(X,Y)$.

\section{Maximal Matroids and Matroid Elevations}
 
Let $\cX$ be a family of  subsets of a finite set $E$.  
We first derive a sufficient condition for a given $\cX$-matroid on $E$ to be the unique maximal such matroid.
We then use results from \cite{CJT2} to construct a maximal element in the poset of all $\cX$-matroids on $E$ (whenever this poset is non-empty).
%

\subsection{A sufficient condition for unique maximality}

Recall that a set $F$ in a matroid $\cM$ is {\em connected} if, for every pair of elements $e_1,e_2\in F$, there exists a circuit $C$ of $\cM$ with $e_1,e_2\in C\subseteq F$, and that $F$ is a {\em connected component} of $\cM$ if $F$ is either a coloop of $\cM$ or a maximal connected set in $\cM$. It is well known that the set $\{F_1,F_2,\ldots, F_m\}$ of all connected components partitions the ground set of $\cM$ and that $\rank \cM=\sum_{i=1}^mr_\cM(F_i)$. In addition, $F$ is connected in $\cM$ if and only if $r_\cM(F)<r_\cM(F')+r_\cM(F'')$ for all partitions $\{F',F''\}$ of $F$.


\begin{lemma}\label{lem:conflat}
Let  
$\cX$ be a family of subsets of a finite set $E$ and 
$\cM$ be a loopless $\cX$-matroid on $E$.
Suppose that,
for every connected flat $F$ of $\cM$, there is a proper $\cX$-sequence  $\cS$
with  $r_\cM(F)={\rm val}(F,\cS)$.
Then ${\rm val}_{\cX}=r_{\cM}$ 
and 
$\cM$ is the unique maximal $\cX$-matroid on $E$. 
\end{lemma}
\begin{proof}
Since  $r_{\cM}\leq \val_\cX$ for all $\cX$-matroids on $E$ by Lemma~\ref{lem:upper1}, it will suffice to show that, for each $F\subseteq E$, 
there is a proper $\cX$-sequence $\cS$  such that $r_{\cM}(F)={\rm val}(F,\cS)$. 

Suppose, for a contradiction, that this is false for some set $F$. 
We may assume that $F$ has been chosen such that $r_\cM(F)$ is as small as possible  and, subject to this condition, $|F|$ is as large as possible.
If $F$ is not a flat then $r_\cM(F+e)=r_\cM(F)$ for some $e\in E\sm F$ and we can now use the maximality of $|F|$ to deduce that there exists a proper  $\cX$-sequence $\cS$  such that 
$r_{\cM}(F+e)={\rm val}(F+e,\cS)$.
By Lemma~\ref{lem:upper1} and $r_{\cM}(F)=r_{\cM}(F+e)$, $e\in \bigcup_{X\in \cS} X$.
Hence,  ${\rm val}(F+e,\cS)={\rm val}(F,\cS)=r_{\cM}(F+e)=r_{\cM}(F)$. This would contradict the choice of $F$. Hence $F$ is a flat.

Suppose $F$ is not connected. Then we have $r_\cM(F)=r_\cM(F_1)+r_\cM(F_2)$ for some partition $\{F_1,F_2\}$ of $F$. Since 
$\cM$ is loopless,
$F_i$ is a flat of $\cM$ and  $1\leq r_\cM(F_i)<r_\cM(F)$ for both $i=1,2$. The choice of $F$ now implies that there exists 
a proper $\cX$-sequence $\cS_i$  such that $r_{\cM}(F_i)={\rm val}(F_i,\cS_i)$ for $i=1,2$. Since each $F_i$ is a flat, we have $X_i\subseteq F_i$ for all $X_i\in \cS_i$ by  Lemma~\ref{lem:upper1}. This implies that the concatenation  $\cS=(\cS_1,\cS_2)$ is a proper $\cX$-sequence and satisfies
$${\rm val}(F,\cS)={\rm val}(F_1,\cS_1)+{\rm val}(F_2,\cS_2)=r_{\cM}(F_1)+r_{\cM}(F_2)=r_{\cM}(F).$$ 
This contradicts the choice of $F$. 

Hence 
$F$ is a connected flat and we can use  the hypothesis of the lemma to deduce that there is a proper $\cX$-sequence $\cS$  such that $r_{\cM}(F)={\rm val}(F,\cS)$, as required.
\end{proof}

\subsection{Matroid elevations}

The {\em truncation} of a matroid $\cM_1=(E,\cI_1)$ of rank $k$ is the matroid $\cM_0=(E,\cI_0)$ of rank $k-1$, where $\cI_0=\{I\in \cI_1\,:\,|I|\leq k-1\}$.
Crapo~\cite{C} defined {\em matroid erection} as the `inverse operation' to truncation. So $\cM_1$ is an {\em erection} of $\cM_0$ if $\cM_0$ is the truncation of $\cM_1$. (For technical reasons we also consider $\cM_0$ to be a {\em trivial erection} of itself.)
Note that, although every matroid  has a unique truncation,  matroids may have several, or no,  non-trivial erections.

Crapo~\cite{C} showed that the poset of all erections of a matroid $\cM_0$ is actually a {lattice}. It is clear that the trivial erection of $\cM_0$ is the unique minimal element in this lattice. Since this is a finite lattice, there also exists a unique 
maximal element which Crapo called the  {\em free erection of $\cM_0$}.  

A {\em partial elevation} of  $\cM_0$ is any matroid $\cM$ which can be constructed from $\cM_0$ by a  sequence of erections. A {\em (full) elevation} of $\cM_0$ is a partial elevation $\cM$ which has no non-trivial erection.
The {\em free elevation} of $\cM_0$ is the matroid we get from $\cM_0$ by recursively constructing a sequence of 
free erections until we arrive at a matroid which has no non-trivial erection.
The set  of all partial elevations of $\cM_0$ forms a poset $P(\cM_0)$ under the weak order and $\cM_0$ is its unique minimal element.
Every maximal element of $P(\cM_0)$ will have no non-trivial erection so will be a full elevation of $\cM_0$. 
Given Crapo's result that the poset of all erections of $\cM_0$ is a lattice, it is tempting to conjecture that $P(\cM_0)$ will also be a lattice and that the free elevation of $\cM_0$ will be its unique maximal element. But this is false:  Brylawski  gives a counterexample based on the Vamos matroid in \cite{B} and we will construct another counterexample using $\cH$-matroids on $K_n$ in Section \ref{sec:notunique}. 
The following weaker result is given in  \cite{CJT2}.

\begin{lemma}[{\cite[Lemma 3.1]{CJT2}}]\label{lem:free_elevation}
Suppose that $\cM_0$ is a matroid. Then the free elevation of $\cM_0$ is a maximal element in the poset of all partial elevations of $\cM_0$.
\end{lemma}

Our next result extends Lemma \ref{lem:free_elevation} to $\cX$-matroids. Given a finite set $E$ and an integer $k$, let $\cU_k(E)$ 
be the {\em uniform matroid on $E$ of rank $k$}, i.e. the matroid on $E$ in which a set $F\subseteq E$ is independent if and only if $|F|\leq k$.

\begin{lemma}\label{lem:free_elevation0}
Let $E$ be a finite set, 
$\cX$ be a family of subsets of  $E$ of size at most $s$, 
and   $\cM_0$ be a maximal matroid in the poset of all $\cX$-matroids on $E$ with rank at most $s$.  
Suppose that $\cM_0\neq \cU_{s-1}(E)$.
Then the free elevation of $\cM_0$ is a maximal matroid in the poset of all $\cX$-matroids on $E$.
\end{lemma}
\begin{proof} 
Let $\cM$ be the free elevation of $\cM_0$. 
Since $\cM_0$ is an $\cX$-matroid and $\cM_0\neq \cU_{s-1}(E)$,
every set in $\cX$ is a non-spanning circuit of $\cM_0$. This implies that 
every  partial elevation of $\cM_0$ is an $\cX$-matroid. In particular, $\cM$ is an $\cX$-matroid.

Lemma \ref{lem:free_elevation} implies that $\cM$ is a maximal element in the poset of all partial elevations of $\cM_0$. 
Let $\cN$ be an $\cX$-matroid on $E$ which is not a partial elevation of $\cM_0$.
 Let $\cN_0$ be the truncation of $\cN$ to rank $s$ if $\cN$ has rank at least $s$, and otherwise let $\cN_0=\cN$.  
 Then $\cN_0\neq \cM_0$.  Since $\cM_0$ is a maximal $\cX$-matroid in the poset of all $\cX$-matroids on $E$ with rank at most $s$,  $\cN_0\not \succ\cM_0$ holds.
 Hence there exists $F\subseteq E$ with the properties that $|F|\leq s$, $F$ is dependent in $\cN_0$ and $F$ is independent in $\cM_0$. This implies that $F$ is dependent in $\cN$ and  independent in $\cM$ so $\cN\not \succ\cM$. Hence $\cM$ remains as a maximal element in the poset of all $\cX$-matroids on $E$.
\end{proof}

Lemma~\ref{lem:free_elevation0} can be applied whenever there exists at least one $\cX$-matroid $\cM$ on $E$ since we can truncate $\cM$ to obtain an $\cX$-matroid of rank at most $s$, and hence the poset of all $\cX$-matroids on $E$ with rank at most $s$ will be non-empty.\footnote{Note that in our main motivation, the Edmonds Problem, $\cX$ will be a family of minimal row dependencies of a matrix $A$ and hence the row matroid of $A$ will be an $\cX$-matroid.} 
In the next subsection, we give an explicit construction of  a maximal  $\cX$-matroid in the poset of all $\cX$-matroids on $E$ with rank at most $s$ whenever $\cX$ is an $s$-uniform families.

%

We close this subsection by stating a useful property of free elevations. We say that an $\cX$-matroid $\cM$ on a finite set $E$ has the {\em $\cX$-covering property} if 
every cyclic flat in $\cM$ is the union of sets in $\cX$.

\begin{lemma}[{\cite[Lemma 3.6]{CJT2}}]\label{lem:cover}
Let $\cM_0$ be a matroid on a finite set $E$ and $\cX$ be the family of non-spanning circuits of $\cM_0$.
Suppose that $E=\bigcup_{X\in \cX} X$.
Then the free elevation of $\cM_0$ has the $\cX$-covering property.
\end{lemma}

\subsection{Uniform $\cX$-matroids}
%

A family $\cX$ of sets  is  {\em $k$-uniform} if each set in $\cX$ has size $k$.
Given a $k$-uniform family $\cX$, 
the {\em $\cX$-uniform  system  ${\cal U}_\cX$} is defined as the pair $(E, {\cal I}_{\cX})$, where $E=\bigcup_{X\in\cX}X$ and 
\[
{\cal I}_{\cX}:=\{F\subseteq E:  |F|\leq k \text{ and } F\notin \cX\}.
\]
We first characterise when ${\cal U}_{\cX}$ 
is a matroid.
We say that  the $k$-uniform family $\cX$ is {\em union-stable} if,
for any  $X_1, X_2\in \cX$ and $e\in X_1\cap X_2$, 
either $|(X_1\cup X_2)-e|>k$ or   $(X_1\cup X_2)-e\in \cX$.

\begin{lemma}\label{lem:Xstable}
Suppose that $\cX$ is  a $k$-uniform  family.
Then $\cU_\cX$ is a matroid if and only if $\cX$ is union-stable.
\end{lemma}
\begin{proof}
Let $\cC=\cX\cup \{C\subseteq E:  |C|= k \text{ and  } X\not\subseteq C \text{ for all }X\in\cX\}$. It is straightforward to check $\cU_X$ is a matroid if and only if $\cC$ satisfies the matroid circuit axioms and that the latter property holds if and only if $\cX$ is union-stable.
%
%
\end{proof}

Given an arbitrary $k$-uniform family $\cX$, we construct the {\em union-stable closure $\bar \cX$  of $\cX$} by first putting $\bar \cX=\cX$ and then recursively adding $(X_1\cup X_2)-e$ to $\bar \cX$ whenever $X_1,X_2\in \bar \cX$, $|X_1\cup X_2|=k+1$ and $e\in X_1\cap X_2$. It is straightforward to check that the resulting family $\bar \cX$ is $k$-uniform and union-stable and that  $\cU_{\bar\cX}$ is a maximal matroid in   the poset of all $\cX$-matroids on $E$ with rank at most $k$. We can now apply Lemma \ref{lem:free_elevation0} to deduce:

\begin{lemma}\label{lem:free_elevation1}
Let  $\cX$ be a $k$-uniform family of sets such that
$\cU_{\bar \cX}\neq\cU_{k-1}(E)$.
Then the free elevation of $\cU_{\bar \cX}$ is a maximal $\cX$-matroid on $E$.
\end{lemma}
%
\noindent Note that if $\cU_{\bar \cX}= \cU_{k-1}(E)$ then $\cU_{\bar \cX}$ is the unique maximal $\cX$-matroid on $E$ but the free-elevation of $\cU_{\bar \cX}$ is the {\em free matroid on $E$} i.e. the matroid in which every subset of $E$ is independent.

Suppose that  $G$ and $H$ are graphs with $|E(H)|=k$. We will also assume that every edge of $G$ belongs to a subgraph  which is isomorphic to  $H$ (we can reduce to this case by deleting all edges of $G$ which do not belong to copies of $H$). Recall that  $\{H\}_G$ denotes the $k$-uniform family containing all edge sets of copies of $H$ in $G$.
The graph $H$ is said to be {\em union-stable} on  $G$ if $\{H\}_G$ is union-stable, i.e., for any two distinct copies $H_1$ and $H_2$ of $H$ in $G$  and any $e\in E(H_1)\cap E(H_2)$, 
either $H_1\cup H_2-e$  is isomorphic to $H$ or $|E(H_1\cup H_2-e)|>k$. To simplify notation we denote the uniform $\{H\}_G$-matroid  $\cU_{\{H\}_G}$ by $\cU_H(G)$ when $H$ is union-stable.
Examples of union-stable graphs on $K_n$ are stars, cycles, complete graphs, and complete bipartite graphs.
Lemmas~\ref{lem:free_elevation0} and \ref{lem:Xstable} now give:

\begin{lemma}\label{lem:stable}
Suppose that $G$  and $H$ are graphs.
Then $\cU_H(G)$ is a matroid if and only if $H$ is union-stable on $G$.
Furthermore, if $\cU_H(G)$ is a matroid, then its free elevation is a maximal $H$-matroid on $G$.
\end{lemma}

\section{Weakly Saturated Sequences}
Let   $\cX$ be a family of subsets of a finite set $E$, and $F_0\subseteq E$.
Recall that a proper $\cX$-sequence $(X_1, X_2,  \dots, X_m)$ is a  weakly $\cX$-saturated sequence from $F_0$ if 
$|X_i\setminus (F_0\cup \bigcup_{j<i} X_j)|=1$ for all $i$ with $1\leq i\leq m$.
%
We say that a set $F\subseteq E$ {\em can be constructed by a weakly $\cX$-saturated sequence from $F_0$} if 
there is a weakly $\cX$-saturated sequence 
$\cS$ from $F_0$ with  $F=F_0\cup \bigcup_{X\in \cS} X$. Note that if this is the case then we will have ${\rm val}(F,\cS)= |F_0|$. We can combine this simple observation with Lemma \ref{lem:conflat} to give several examples of unique maximality.

\begin{lemma}\label{lem:uniform}
Let $\cX$ be a  $k$-uniform family of subsets of a finite set $E$.
Suppose that $E$ can be constructed by a weakly $\cX$-saturated sequence from some $X_0\in \cX$.
Then the rank $k-1$ uniform matroid  $\cU_{k-1}(E)$ is the unique maximal $\cX$-matroid on $E$ and its rank function is ${\rm val}_{\cX}$.
\end{lemma}
\begin{proof}
We denote $\cU=\cU_{k-1}(E)$.
Since $\cU$ is uniform, $E$ is the only connected flat in $\cU$ and hence, by Lemma \ref{lem:conflat}, it will suffice to show that
there is a proper $\cX$-sequence $\cS$ 
such that 
$r_{\cU}(E)={\rm val}_{\cX}(\cS, E)$.  
By hypothesis, there is a   weakly saturated $\cX$-sequence  $\cS_0$ from $X_0$ to $E$.
Let $\cS$ be the proper $\cX$-sequence obtained by  inserting $X_0$ at the beginning of  $\cS_0$.
Then 
$${\rm val}_{\cX}(\cS, E)=
{\rm val}_{\cX}(\cS_0, E)-1= |X_0|-1=k-1=r_{\cU}(E),$$
as required.
\end{proof}

The same proof technique can handle a slightly more complicated situation.

\begin{lemma}\label{lem:Xuniform}
Let $\cX$ be a $k$-uniform, union-stable family of sets.
Suppose that $E$ can be constructed by a weakly $\cX$-saturated sequence from some $Y\subseteq E$
with $|Y|=k$ and $Y\notin \cX$.
Then $\cU_{\cX}$ is the unique maximal $\cX$-matroid on $E$
and its rank function is ${\rm val}_{\cX}$ .
\end{lemma}
%
\begin{proof}
By Lemma \ref{lem:conflat}, it will suffice to show that
there is a proper $\cX$-sequence $\cS$ 
such that 
$r_{\cU_\cX}(F)={\rm val}_{\cX}(\cS, F)$ for every connected flat  in $\cU_\cX$.  Let $F$ be a connected flat  in $\cU_\cX$.
Then the definition of $\cU_\cX$ implies that the rank of $F$ is either $k$ or $k-1$.

%
%

Suppose that the rank of $F$ is $k$. Then we have $F=E$.
By hypothesis, $E$ can be constructed from $Y$ by a weakly $\cX$-saturated sequence $\cS$.
Then
$r_{\cU_\cX}(E)=k=|Y|={\rm val}(E,\cS)$ follows.

Hence we may assume that the rank of $F$ is $k-1$. Since $F$ is a flat in $\cU_{\cX}$, 
every subset of  $F$ of size $k$ belongs to $\cX$.
We will use this fact to define a weakly $\cX$-saturated sequence for $F$. 
Choose a set $F_0$ of $k-1$ elements in $F$,
and let $X_e=F_0\cup \{e\}$ for each $e\in F\setminus F_0$.
Then each $X_e\in \cX$, and $\{X_e: e\in F\setminus F_0\}$ (ordered arbitrarily) is  a weakly $\cX$-saturated sequence $\cS'$ which constructs  $F$ from $F_0$. 
We have $\val(\cS', F)=
|F_0|=k-1=r_{\cU_\cX}(F)$ as required.
\end{proof}



\subsection*{Applications to matroids on graphs.}

Given graphs $G$ and $H$ and subgraphs $F_0,F\subseteq G$, we say that  
{\em $F$ can be constructed by a weakly $H$-saturated sequence from $F_0$} if $E(F)$ can be constructed by a weakly $\{H\}_G$-saturated sequence from $E(F_0)$.

\begin{lemma}\label{lem:matching}
Let $H_k$ be the vertex-disjoint union of $k$ copies of $K_2$. 
Then $K_n$ can be constructed by a weakly saturated $H_k$-sequence from any copy of $H_k$ in $K_n$ whenever $n\geq 2k+1$.
\end{lemma}
\begin{proof}
Let $H=\{e_1,e_2,\ldots,e_k\}$ be a copy of $H_k$ in $K_n$ for some $n\geq 2k+1$. We show that $K_n$ has a weakly saturated $H_k$-sequence starting from $H$ by induction on $n$. Choose a vertex $v\in V(K_n)\sm V(H)$.

Suppose $n=2k+1$. For each edge $f$ from $v$ to $H$ we can choose a $k$-matching $H^f$ containing $f$ and $k-1$ edges of $H$. Then for each edge $g$ of $(K_n-v)-E(H)$ we can choose a $k$-matching $H^g$ containing $g$ and $k-1$ edges of
$H\cup \bigcup_{f \sim v}H^f$. 
Concatenating $H$ with $H^f$ for $f\sim v$ and then $H^g$ for the remaining edges $g$ gives a
weakly saturated $H_k$-sequence which constructs $K_{2k+1}$ from $H$.

Now suppose $n>2k+1$. By induction, $K_n-v$ has a weakly saturated $H_k$-sequence $\cS$ starting from $H$. For each edge $f$ from $v$ to $H$ we can choose a $k$-matching $H^f$ containing $f$ and $k-1$ edges of $K_n-v$.  Concatenating $H$ with $H^f$  gives the required weakly saturated $H_k$-sequence for $K_n$.
\end{proof}

Combining Lemmas~\ref{lem:uniform} and \ref{lem:matching}, we immediately obtain:
\begin{theorem}
Let $H_k$ be the vertex-disjoint union of $k$ copies of $K_2$.
Then $\cU_{k-1}(K_n)$ is the unique maximal $H_k$-matroid on $K_n$ for all $n\geq 2k+1$, and $\val_{H_k}$ is its rank function.
\end{theorem}

Let $P_k$ be the path with $k$ edges. It is straightforward to show that $K_n$ can be constructed by a  weakly saturated $P_k$-sequence starting from a particular copy of $P_k$ in $K_n$ whenever $n\geq k+1$. Lemma~\ref{lem:uniform} now gives:

\begin{theorem}\label{thm:path}
Let $P_k$ be the path of length $k$.
Then $\cU_{k-1}(K_n)$ is the unique maximal $P_k$-matroid on $K_n$ for all $n\geq k+1$, and $\val_{P_k}$ is its rank function.
\end{theorem}

Sitharam and Vince \cite{SV} showed that $\cU_{K_{1,3}}(K_n)$ is the unique maximal $K_{1,3}$-matroid on $K_n$ for all $n\geq 4$. Their result can be deduced from Lemma \ref{lem:Xuniform} 
since $K_{1,3}$ is union-stable and $K_n$ can be constructed by a weakly $K_{1,3}$-saturated sequence starting from a copy of $K_3$. We may also deduce that $\val_{K_{1,3}}$ is the rank function of $\cU_{K_{1,3}}(K_n)$.

It is an open problem  to determine whether  there is a unique maximal $T_k$-matroid on $K_n$   for any fixed tree $T_k$ with  $k$ edges. 
The following result gives some information on this poset: it implies that the rank of every $T_k$-matroid on $K_n$ is bounded by a quadratic polynomial in $k$.
  
\begin{lemma}\label{lem:bound}
Suppose $H$ is a graph with $s$ vertices and minimum degree  $\delta$, and $\cM$ is a  $H$-matroid on $K_n$ with $n\geq s-1$. Then the rank of $\cM$ is at most $(\delta-1) (n-s+1)+{s-1\choose 2}$.
\end{lemma}
\begin{proof}
Let $V(K_n)=\{v_1,v_2,\dots, v_n\}$ and put $V_i=\{v_1, v_2,\dots, v_i\}$ and $E_i=E(K[V_i])$ for $1\leq i\leq n$.
Choose a base $B_{s-1}$ of $E(K(V_{s-1}))$ in $\cM$. 
Clearly $|B_{s-1}|\leq {s-1\choose 2}$.  

 We inductively construct a base $B_i$ of $E_i$ in $\cM$ from $B_{s-1}$, for $i=s,\dots, n$.
Suppose we have a base $B_i$ of $E_i$,
and let $B_{i+1}'=B_i\cup\{v_{i+1}v_1, \dots, v_{i+1}v_{\delta-1}\}$.
Then, for each $j$ with $s\leq j\leq i$,  
$K(V_i)+\{v_{i+1}v_1, \dots, v_{i+1}v_{\delta-1}, v_{i+1} v_j\}$ contains a graph isomorphic to $H$ in which
the degree of $v_{i+1}$ is equal to $\delta$.
Hence $B_{i+1}'$ spans $E_i$ in $\cM$.
Let $B_{i+1}$ be a base of $B_{i+1}'$ obtained by extending $B_i$. 
Then $B_{i+1}$ is obtained from $B_i$ by adding at most $\delta-1$ edges.
This implies that $B_n$  has size at most $(\delta-1) (n-s+1)+{s-1\choose 2}$.
The lemma now follows since $B_n$ is a base of $\cM$.
\end{proof}

We close this section with one more application of Lemma~\ref{lem:Xuniform}. Consider the graph   $G_5$ in Figure~\ref{fig:K23}. 
It is straightforward to check that $G_5$ is union-stable and that $K_n$ can be constructed by  a weakly saturated $G_5$-sequence starting from $K_{2,3}$. Hence $\cU_{G_5}(K_n)$ is the unique maximal $G_5$-matroid on $K_n$ for all $n\geq 5$, and $\val_{G_5}$ is its rank function.

\section{Matroids Induced by Submodular Functions}
\label{sec:induced}
%

In this section we use weakly saturated sequences and a matroid construction due to Edmonds 
to give more examples of unique maximal matroids. 

\begin{theorem}[Edmonds \cite{E}]\label{thm:edmonds}
Let  $E$ be a finite set and $f:2^E\rightarrow \mathbb{Z}$ be a non-decreasing, submodular function. Put
\[
{\cal I}_f:=\{F\subseteq E: |I|\leq f(I) \text{ for any $I\subseteq F$ with $I\neq \emptyset$} \}.
\]
Then $\cM_f:=(E, {\cal I}_f)$ is a matroid with rank function $\hat f:2^E\rightarrow \mathbb{Z}$ given by
\begin{equation}\label{eq:dilworth}
\hat{f}(F):=\min\left\{|F_0|+\sum_{i=1}^k f(F_i) : \text{$F_0\subseteq F$ and  $\{F_1,\dots, F_k\}$ is a partition  of $F\sm F_0$}\right\}.
\end{equation}
\end{theorem}

We refer to the matroid $\cM_f$ given by Edmond's theorem as the {\em matroid induced by $f$}. Given a  set $F\subseteq E$ it is straightforward to check that:
\begin{equation}
\label{eq:induced1}
\begin{split}
&\mbox{$F$  is a circuit in $\cM_f$ if and only if $0\neq |F|=f(F)+1$ and}\\
&\mbox{$|F'|\leq f(F')$  for all $F'\subseteq F$ with $\emptyset \neq F'\neq F$;}
\end{split}
\end{equation}
\begin{equation}
\label{eq:induced2}
\mbox{$F$  is a flat in $\cM_f$ if and only if $f(F+e)=f(F)+1$   for all $e\in E\sm F$.}
\end{equation}

The function $f_{a,b}:2^{E(K_n)}\to\ \mathbb{Z}$ by $f_{a,b}(F)=a|V(F)|-b$ is submodular and non-decreasing for all $a,b\in  \mathbb{Z}$ with $a\geq 0$ and hence induces a matroid $\cM_{f_{a,b}}(K_n)$ on $E(K_n)$. These matroids are known as {\em count matroids}.  It is well known that the cycle matroid of $K_n$ is the count matroid $\cM_{f_{1,1}}(K_n)$.
Another well-known example is when $a=2$ and $b=3$, which gives the rigidity matroid of generic  frameworks in $\R^2$. 
Sitharam and Vince~\cite{SV}  showed that 
$\cM_{f_{1,1}}(K_n)$ and $\cM_{f_{2,3}}(K_n)$ are the unique maximal $K_3$-matroid and $K_4$-matroid on $K_n$, respectively. Slightly weaker versions of these results were previously obtained by Graver \cite{G}. 

We will show that the maximality of both these matroids, as well as that of  several other count matroids, follow easily from   Lemma~\ref{lem:conflat} and Theorem~\ref{thm:edmonds}. We need the following observation on the connected flats of count matroids which follows immediately from (\ref{eq:induced1}) and (\ref{eq:induced2}).

\begin{lemma}\label{lem:countflat} 
Suppose $a,b\in  \mathbb{Z}$ with $a\geq 0$ and $F\subseteq E(G)$ is a connected flat in  $\cM_{f_{a,b}}(G)$. Then $G[F]$ is the subgraph of $G$ induced by $V(F)$ and $|F|\geq  a|V(F)|-b+1$. 
\end{lemma}
 
Let $K_n^-$ be the graph obtained from $K_n$ by removing an edge.
\begin{theorem}\label{thm:K3}
(a) $\cM_{f_{1,1}}(K_n)$  is the unique maximal $K_3$-matroid on $K_n$
 and its rank function is $\val_{K_3}$. \\
(b) $\cM_{f_{2,3}}(K_n)$  is the unique maximal $K_4$-matroid on $K_n$ and its rank function is $\val_{K_4}$.
\\
(c) $\cM_{f_{1,0}}(K_n)$  is the unique maximal $K_4^-$-matroid on $K_n$ and its rank function is $\val_{K_4^-}$.
  \\
(d)
 $\cM_{f_{2,2}}(K_n)$  is the unique maximal $K_5^-$-matroid on $K_n$ and its rank function is $\val_{K_5^-}$.
\\ 
(e) $\cM_{f_{3,5}}(K_n)$  is the unique maximal $K_6^-$-matroid on $K_n$ and its rank function is $\val_{K_6^-}$.
 \end{theorem}
 \begin{proof} In each case   
 $\cM_{f_{a,b}}(K_n)$ is loopless and is an $X$-matroid on $K_n$ for $X=K_3$ , $K_4$, $K_4^-$, $K_5^-$, $K_6^-$, respectively, by (\ref{eq:induced1}).
Lemmas \ref{lem:conflat} and \ref{lem:countflat}  will now imply that $\cM_{f_{a,b}}(K_n)$  is the unique maximal $X$-matroid on $K_n$ once we have shown that,
for every $K_m\subseteq K_n$ with $|E(K_m)|>am-b$, there is a proper $X$-sequence $\cS$ with 
$r_{\cM_{f_{a,b}}}(K_m)={\rm val}(K_m,\cS)$. We will do this by finding   a weakly saturated $X$-sequence which constructs $K_m$ from a subgraph $G\subset K_m$
with  $|E(G)|=am-b$. Let $V(K_m)=\{v_1,v_2,\ldots,v_m\}$.

In cases (a) and (b) we can use the well known fact that, for $m\geq d+1$, $K_m$ can be constructed by a weakly saturated $K_{d+2}$-sequence starting from the spanning subgraph $G$ with  $E(G)=\{v_iv_j: 1\leq i<j\leq d\}\cup \{v_iv_j: 1\leq i\leq d, d+1\leq j\leq m\}$, then taking $d=3,4$ for  cases (a), (b), respectively.
  
In cases (c), (d) and (e) we can use a result of 
Pikhurko~\cite{P2} that, for $d\geq m+1$, 
$K_m$ has a weakly saturated $K_d^-$-sequence starting from 
a spanning subgraph with $(d-3)m-{d-2\choose 2}+1$ edges, and then taking $d=4,5,6$ for cases (c), (d), (e), respectively.
%
%
%
%
%
 \end{proof}

Lemma~\ref{lem:conflat} can also be used to extend Theorem~\ref{thm:K3} to matroids on non-complete graphs. For example, if $G$ is a 
chordal graph, 
then every connected flat of $\cM_{f_{1,1}}(G)$ is a 2-connected chordal graph and we can use  Lemma~\ref{lem:conflat} and an appropriate weakly saturated $K_3$-sequence to deduce that $\cM_{f_{1,1}}(G)$ is  the unique maximal $K_3$-matroid on $G$ and $\val_{K_3}$ is its rank function.

Our next result gives another example of uniqueness for matroids on non-complete graphs.

\begin{theorem}\label{thm:K23Kmn} 
The matroid $\cM_{f_{1,0}}(K_{m,n})$ is the unique maximal $K_{2,3}$-matroid on $K_{m,n}$ and $\val_{K_{2,3}}$ is its rank function.
\end{theorem}
\begin{proof}
By (\ref{eq:induced1}) and (\ref{eq:induced2}), each copy of $K_{2,3}$ is a circuit in 	$\cM_{f_{1,0}}(K_{m,n})$ and each connected flat is a copy of $K_{s,t}$ for some $s\geq 2, t\geq 3$.
By the same argument 
as in the proof of Theorem~\ref{thm:K3}, it will suffice to show that,
for any $K_{s,t}$ with $s\geq 2, t\geq 3$, 
 there is a weakly saturated $K_{2,3}$-sequence which constructs $K_{s,t}$ from a subgraph $G\subset K_{s,t}$ with $|E(G)|=s+t$. This follows easily
by taking $V(G)=\{u_1,u_2,\ldots,u_s\}\cup \{w_1,w_2,\ldots,w_t\}$ and $E(G)=\{u_2w_2\}\cup\{u_1w_i: 1\leq i\leq t\}\cup \{u_iw_1: 2\leq i\leq s\}$. 
\end{proof}

In contrast to this result, we will see in Section 5 that there are two distinct maximal $K_{2,3}$-matroids on $K_n$.

The {\em even cycle matroid} is the matroid $\cM$ on $E(K_n)$, 
 in which a set $F$ is independent if and only if each connected component of the induced subgraph $K_n[F]$ contains at most one cycle, and this cycle is odd if it exists. The rank function of  $\cM$ is given by $r_\cM(F)=|V(F)|-\beta(F)$, where $\beta(F)$ denotes the number of bipartite connected components in the graph $K_n[F]$. 
We can use this fact to define a modified version of count matroids. 

For
$a,b,c\in  \mathbb{Z}$, define $g_{a,b,c}:2^{E(K_n)}\to \mathbb{Z}$ by $g_{a,b,c}(F)=a|V(F)|-b\beta(F)-c$. Then $g_{a,b,c}$ is submodular and non-decreasing for all $a,b\in  \mathbb{Z}$ with $a\geq b\geq 0$ since the functions $F\mapsto |V(F)|$ and $F\mapsto |V(F)|-\beta(F)$ are both submodular and non-decreasing.
Hence $g_{a,b,c}$ induces a matroid $\cM_{g_{a,b,c}}(K_n)$ on $E(K_n)$ whenever $a\geq b\geq 0$. 
We will give  examples of families $\cH$ for which  $\cM_{g_{a,b,c}}(K_n)$ is the unique maximal $\cH$-matroid on $K_n$. We need the following observation on the connected flats of $\cM_{g_{a,b,c}}(K_n)$ which follows immediately from (\ref{eq:induced1}) and (\ref{eq:induced2}).

\begin{lemma}\label{lem:countflatbip} 
Suppose $a,b,c\in  \mathbb{Z}$ with $a\geq b\geq  0$, $c\geq 0$,  and $F\subseteq E(K_n)$ is a connected flat in  $\cM_{g_{a,b,c}}(K_n)$. Then $K_n[F]$ is either a complete graph with $|F|\geq a|V(F)|-c+1$ or a complete bipartite graph with $|F|\geq a|V(F)|-b-c+1$. 
\end{lemma}

The hypothesis of Lemma \ref{lem:countflatbip} that $c\geq 0$ is needed to ensure that the circuits of  $\cM_{g_{a,b,c}}(K_n)$ induce connected subgraphs of $K_n$, which in turn implies that the same property holds for the connected flats of  $\cM_{g_{a,b,c}}(K_n)$. This is not true when $c\leq -1$, for example the disjoint union of two copies of $C_4$ is both a circuit and a connected flat in  $\cM_{g_{1,1,-1}}(K_n)$.

\begin{theorem}\label{thm:countbip}
(a) The even cycle matroid $\cM_{g_{1,1,0}}(K_n)$ is the unique maximal $C_4$-matroid on $K_n$ and its rank function is $\val_{C_4}$.
\\
(b) $\cM_{g_{2,1,2}}$ is the unique maximal $\{K_5^-, K_{3,4}\}$-matroid on $K_n$ and its rank function is $\val_{\{K_5^-, K_{3,4}\}}$.
\end{theorem}

 \begin{proof} In each case,   
 $\cM_{g_{a,b,c}}(K_n)$ is loopless and is an $\cX$-matroid  on $K_n$ for
 $\cX=\{C_4\}_{K_n}$ and $\cX=\{K_5^-, K_{3,4}\}_{K_n}$, respectively, by (\ref{eq:induced1}). 
Lemmas \ref{lem:conflat} and \ref{lem:countflatbip}  will now imply that $\cM_{f_{a,b}}(K_n)$  is the unique maximal $\cX$-matroid on $K_n$ once we have shown that:
for every $K_m\subseteq K_n$ with $|E(K_m)|\geq am-c+1$, there is a weakly saturated $\cX$-sequence which constructs $K_m$ from a subgraph $G\subset K_m$
with  $|E(G)|=am-c$; and for every $K_{s,t}\subseteq K_n$ with $|E(K_{s,t})|\geq am-b-c+1$, there is a weakly saturated $\cX$-sequence which constructs $K_{s,t}$ from a subgraph $G\subset K_{s,t}$
with  $|E(G)|=am-b-c$. 
Let $V(K_m)=\{v_1,v_2,\ldots,v_m\}$ and $V(K_{s,t})=\{u_1,u_2,\ldots,u_s\}\cup \{w_1,w_2,\ldots,w_t\}$.
\\[1mm]
(a) For $m\geq 4$, $K_{m}$ can be constructed by 
a weakly saturated $C_4$-sequence starting from a spanning subgraph  $G$ with
$g_{1,1,0}(E(K_{m}))=m$ edges by taking $E(G)=\{v_2v_3\}\cup \{v_1v_i: 2\leq i\leq m\}$.
For $s,t\geq 2$,  $K_{s,t}$ can be constructed by 
a weakly saturated $C_4$-sequence starting from a spanning subgraph  $G$ with  
$g_{1,1,0}(E(K_{s,t}))=s+t-1$ edges by taking $E(G)=\{u_1w_i: 1\leq i\leq t\}\cup \{u_iw_1: 1\leq i\leq s\}$. 
\\[1mm]
(b)
For $m\geq 5$, $K_{m}$ can be constructed by 
a weakly saturated $K_5^-$-sequence starting from a spanning subgraph  $G$ with 
$g_{2,1,2}(E(K_{m}))=2m-2$ edges by taking $E(G)=\{v_1v_2,v_3v_4\}\cup \{v_iv_j: 1\leq i\leq 2,\, 3\leq j\leq m\}$.
For $s\geq 3$ and $t\geq 4$,  $K_{s,t}$ can be constructed by 
a weakly saturated $K_{3,4}$-sequence starting from a spanning subgraph  $G$ with  
$g_{2,1,2}(E(K_{s,t}))=2(s+t)-3$ edges by taking 
$E(G)=\{u_iw_j: 1\leq i\leq 3, 1\leq j\leq 3\}\cup \{u_iw_j: 1\leq i\leq 2, \, 4\leq j\leq t\}\cup \{u_iw_j: 4\leq i\leq s,\, 1\leq j\leq 2\}$. 
\end{proof}
 
The matroid  $\cM_{g_{2,1,2}}(K_n)$ in Theorem \ref{thm:countbip}(b) is the Dilworth truncation of the union of the graphic matroid and the even cycle matroid. 
It appears in the context of the rigidity of symmetric frameworks in $\R^2$, see for example~\cite{SW}.

\medskip

The concept of count matroids has been extended to hypergraphs~\cite{P1} and to group-labeled graphs~\cite{ike}. The technique in this section can be  adapted to both settings.

\medskip

We close this section with a remark on the poset of all $\{K_4,K_{2,3}\}$-matroids on $K_n$.  It is straightforward to check that $\cM_{g_{1,1,-1}}(K_n)$ is a $\{K_4,K_{2,3}\}$-matroid on $K_n$. But we cannot 
show it is the unique maximal such matroid by using the same proof technique 
 as Theorem~\ref{thm:countbip} since Lemma~\ref{lem:countflatbip} does not hold when $c<0$.
In fact, we will see in Theorem~\ref{thm:K4K23} below that $\cM_{g_{1,1,-1}}(K_n)$ is not the unique maximal $\{K_4,K_{2,3}\}$-matroid on $K_n$.


\section{Examples of Non-uniquness}\label{sec:notunique}

We will give three examples of posets of $\cH$-matroids on $K_n$ in which there is not a unique maximal matroid. We will frequently use the following fact, which follows from the procedure for constructing the free erection of a matroid due to Duke~\cite{D87}, see for example  \cite[Algorithm 1]{CJT2}.

\begin{lemma}\label{lem:sym} If $\cM_0$ is a symmetric matroid on $K_n$, then the free elevation of $\cM_0$  is a symmetric matroid on $K_n$.
\end{lemma}

Gyula Pap~\cite{Pap} observed that the cycle matroid of $K_n$ and the uniform $C_5$-matroid on $K_n$ are two distinct maximal  $C_5$-matroids on $K_n$.  
We can use Lemma \ref{lem:free_elevation1} to show that Pap's example extends to $C_k$ for all $k\geq 5$. 

\begin{theorem}\label{thm:Ck}
There are two distinct maximal $C_k$-matroids on $K_n$ for all $k\geq 5$ and $n\geq {{k-1}\choose{2}}+2$.
\end{theorem}
\begin{proof}
It is straightforward to check that  $C_k$ is union-stable, and hence $\cU_{C_k}(K_n)$ is a matroid. 
Consider the free elevation $\cM$ of $\cU_{C_k}(K_n)$. Lemmas \ref{lem:free_elevation1} and \ref{lem:sym} imply that  $\cM$ is a maximal $C_k$-matroid on $K_n$ and is symmetric.
We will show that $\cM$ contains a circuit $Z$ such that $K_n[Z]$ has minimum degree one. To see this, consider two distinct copies $X$ and $Y$ of $C_k$ such that $X\cap Y$ forms a path of length $k-2$. By the circuit elimination axiom, $(X\cup Y) - e$ contains a circuit $Z$ of $\cM$ for any $e\in X\cap Y$. Since the copy of $C_4$ in $K_n[(X\cup Y) - e]$ is not a circuit in $\cU_{C_k}(K_n)$, it cannot be a circuit in $\cM$. Hence   $( X\cup Y) - e$ contains a circuit $Z$ such that $K_n[Z]$ has minimum degree one.
We may now apply  Lemma \ref{lem:bound} with $H=Z$ to deduce that the rank of $\cM$ is at most ${k-1}\choose{2}$.

The facts that $\cM$ is a maximal $C_k$-matroid and the  cycle matroid of $K_n$ is a $C_k$-matroid of rank $n-1$,  now imply that there are  at least two maximal $C_k$-matroids on $K_n$ whenever $n\geq {{k-1}\choose{2}}+2$.
\end{proof}
 
W saw in Theorem~\ref{thm:K23Kmn} that the poset of all $K_{2,3}$-matroids on $K_{m,n}$ has a unique maximal element. We next show that this statement becomes false if we change the ground set to $E(K_n)$.
 
 \begin{theorem}\label{thm:K23}
There are two distinct maximal $K_{2,3}$-matroids on $K_n$ for all $n\geq 7$.
\end{theorem}
\begin{proof}
Since $K_{2,3}$ is union-stable, $\cU_{K_{2,3}}(K_n)$ is a matroid by Lemma~\ref{lem:stable},
and its free elevation $\cM$ is a maximal $K_{2,3}$-matroid on $K_n$ and is symmetric  by Lemmas \ref{lem:free_elevation1} and \ref{lem:sym}. We will show that $\cU_{K_{2,3}}(K_n)$ has no non-trivial erection and hence $\cM=\cU_{K_{2,3}}(K_n)$.

We first show that $\cM$ contains a circuit $Z$ such that $K_n[Z]$ has minimum degree one.
To see this consider the graphs $G_1$ and $G_2$ given in Figure~\ref{fig:K23}.
Both are isomorphic to $K_{2,3}$, and hence, by the circuit elimination axiom,
the edge set of $G_3=(G_1\cup G_2)-v_2v_5$, is dependent in $\cM$. 
Since every set of six edges which does not induce a copy $K_{2,3}$  is independent in $\cM$, 
the edge set of $G_3$ is a circuit in $\cM$.
Since the graph $G_4$ in Figure~\ref{fig:K23} is isomorphic to $G_3$, the circuit elimination axiom now implies
the edge set of $G_5=(G_3\cup G_4)-v_4v_5$ is dependent  in $\cM$.
Again, since every set of six edges which does not induce a copy $K_{2,3}$ is independent in $\cM$, 
the edge set of $G_5$ is a circuit  in $\cM$.

\begin{figure}[t]
\centering
\begin{minipage}{0.3\textwidth}
\centering
\includegraphics[scale=0.6]{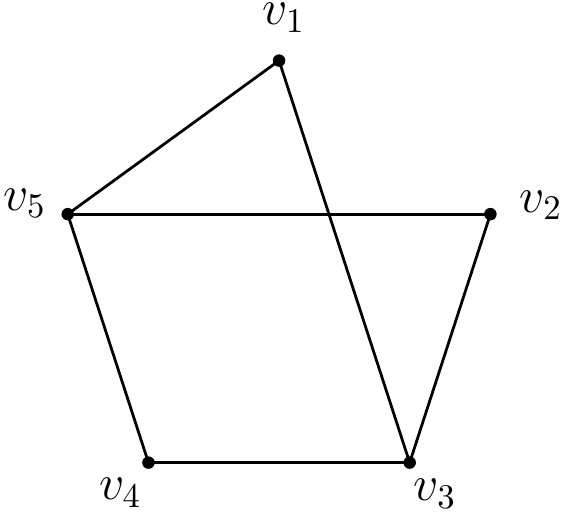}
\par
$G_1$
\end{minipage}
\begin{minipage}{0.3\textwidth}
\centering
\includegraphics[scale=0.6]{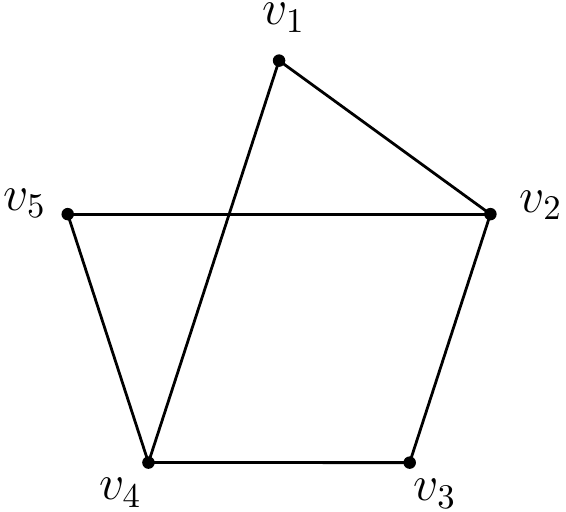}
\par
$G_2$
\end{minipage}
\begin{minipage}{0.3\textwidth}
\centering
\includegraphics[scale=0.6]{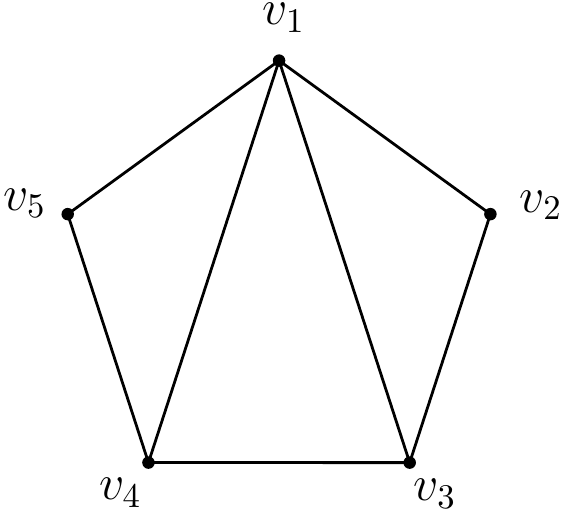}
\par
$G_3$
\end{minipage}

\begin{minipage}{0.3\textwidth}
\centering
\includegraphics[scale=0.6]{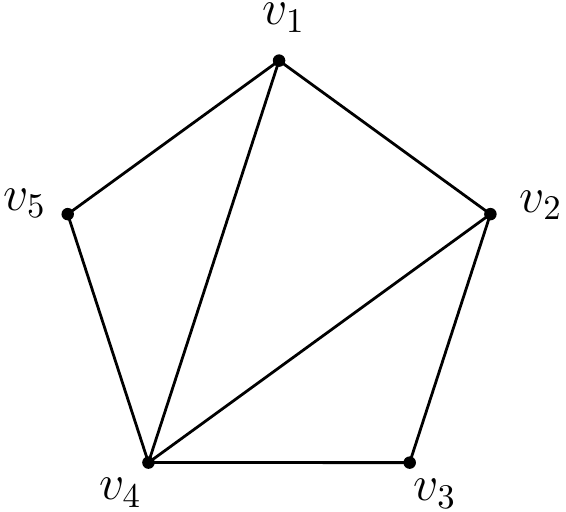}
\par
$G_4$
\end{minipage}
\begin{minipage}{0.3\textwidth}
\centering
\includegraphics[scale=0.6]{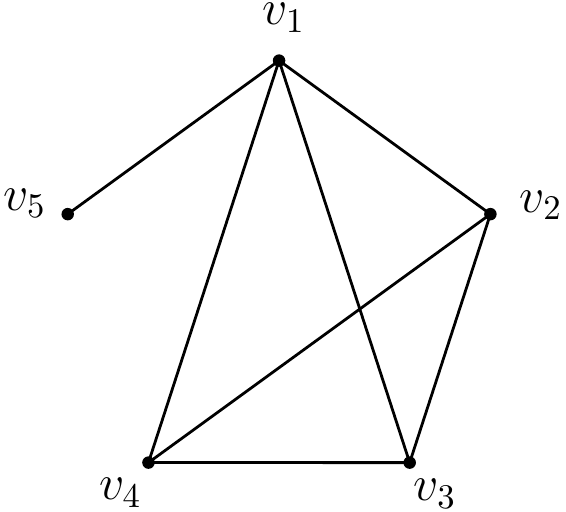}
\par
$G_5$
\end{minipage}
\caption{The graphs $G_1, G_2,\ldots,G_5$ in the proof of Theorem~\ref{thm:K23}.}
\label{fig:K23}
\end{figure}

This implies that $\cM$ is a $G_5$-matroid on $K_n$ and hence, by Lemma~\ref{lem:bound}, the rank of $\cM$ is at most $6$. Since $\cU_{K_{2,3}}(K_n)$ has rank 6, this gives  $\cM=\cU_{K_{2,3}}(K_n)$, and hence
$\cU_{K_{2,3}}(K_n)$ is a maximal $K_{2,3}$-matroid on $K_n$.
Since the bicircular matroid $\cM_{f_{1,0}}$ is a  $K_{2,3}$-matroid on $K_n$ of rank $n$, 
we have at least two maximal $K_{2,3}$-matroids on $K_n$ whenever $n\geq 7$.
\end{proof}


Our final example of this section shows that the unique maximality property may not hold even if we restrict our attention to the poset of all partial elevations of a given $\cH$-matroid on $K_n$ (and hence provides another example, in addition to that given by Brylawski~\cite{B86},  which  shows that the free elevation may not be the unique maximal matroid in the poset of all partial elevations of a given matroid).
Note that the matroids described in Theorem~\ref{thm:Ck} and \ref{thm:K23} do not give such an example.


\begin{theorem}\label{thm:K4K23}
There are two distinct maximal matroids in the poset of all partial elevations of the uniform $\{K_4,K_{2,3}\}$-matroid $\cU_{\{K_4, K_{2,3}\}}(K_n)$ whenever $n\geq 36$.	
\end{theorem}
\begin{proof}
Let $\cX=\{K_4, K_{2,3}\}_{K_n}$. Since $\cX$ is $6$-uniform and union-stable, $\cU_{\cX}$ is a matroid. Hence   
the free elevation $\cM$ of $\cU_{\cX}$ is a symmetric $\cX$-matroid on $K_n$ by Lemma \ref{lem:sym}.
We will show that  $\cM$ has bounded rank.

\begin{claim}\label{claim:K4K232}
The rank of $\cM$ is at most $36$.
\end{claim}
\begin{proof}
Let $D_1$ be the edge set of the union of  a vertex-disjoint 3-cycle and 4-cycle  in $K_n$ and $D_2$ be the edge set of the union of two vertex-disjoint 4-cycles in $K_n$.
We split the proof into three cases.

\paragraph{\boldmath Case 1: $D_1$ is dependent in $\cM$.}
Since $|D_1|=7$ and neither $K_4$ nor $K_{2,3}$ is contained in $K_n[D_1]$, every proper subset of $D$ is independent in $\cM$,
and hence $D_1$ is a circuit in $\cM$.
By Lemma~\ref{lem:cover}, the closure $\cl_\cM(D_1)$ of $D_1$ is the union of copies of $K_4$ and $K_{2,3}$.
Hence  $\cl_\cM(D_1)\neq D_1$ and, for each $e\in \cl_\cM(D_1)\setminus D_1$, 
there exists a circuit $C$ with $e\in C\subsetneq D_1+e$.
Since $K_n[D_1+e]$ cannot contain $K_4$ or $ K_{2,3}$ and $|D_1+e|=8$ we have $|C|=7$. Observe that  any 7-element subset of $D_1+e$ containing $e$ has a vertex of degree one. We may now use Lemma \ref{lem:bound} and the fact that $|V(C)|\leq 9$
to deduce that $\cM$ has rank at most ${{8}\choose{2}}=28$.

\paragraph{\boldmath Case 2:  $D_1$ is  independent in $\cM$ and $D_2$ is  dependent in $\cM$.}
If some proper subset $C$ of $D_2$ is a circuit in $\cM$ then $K_n[C]$ would contain a vertex of degree one and we could again use Lemma \ref{lem:bound} and the fact that $|V(C)|\leq 8$ to deduce that $\cM$ has rank at most ${{7}\choose{2}}=21$. Hence we may assume that $D_2$  is a circuit.

By Lemma~\ref{lem:cover}, $\cl_\cM(D_2)$  is the union of copies of $K_4$ and $K_{2,3}$.
Hence  $\cl_\cM(D_2)\neq D_2$ and, for each $e\in \cl_\cM(D_2)\setminus D_2$, there exists a circuit $C'$ with $e\in C'\subsetneq D_2+e$.
Since  $K_n[D_2+e]$ cannot contain $K_4$ or $K_{2,3}$,  and $|D_2+e|=9$, we have $7\leq |C'|\leq 8$. Observe that every subset of $D_2+e$ of size 7 or 8 which contains $e$ and is distinct from $D_1$, has a vertex of degree one. Hence, we may use Lemma \ref{lem:bound}  and the fact that $|V(C')|\leq 10$ to deduce that $\cM$ has rank at most ${{9}\choose{2}}=36$.

\paragraph{\boldmath
Case 3: $D_1,D_2$ are independent in $\cM$.}
By Theorem~\ref{thm:K23Kmn}, if $F\subseteq  E(K_n)$ induces a bipartite subgraph in $K_n$, then $F$ has rank at most $|V(F)|$ in any $K_{2,3}$-matroid. 
We can use this fact, to compute the rank of the edge set $D_3$ of the graph $G$ in Figure~\ref{fig:K4K23v2}, in $\cM$. 
Let 
$B$ be a base of $D_3$ which contains the independent subset  $D_2$.
Since $D_2+e$ induces a bipartite graph with $|V(D_2+e)|+1$ edges for each $e\in D_3\sm D_2$, $D_2+e$ is dependent. Hence $B=D_2$ and  
$r_{\cM}(D_3)=|D_2|= 8$.

On the other hand the edge set $D_4$ of the 
5-cycle with a chord in $G$ is independent in $\cM$ 
(since $D_4$ is independent 
in $\cU_\cX$). 
Since $r_\cM(D_3)=8$, this implies that 
$D_3-e$ contains a circuit $C$ of $\cM$ with 
$1\leq |C\sm D_4|\leq 3$. Then $K_n[C]$ has a vertex of degree 1, and we may again use Lemma \ref{lem:bound} and the fact that $|V(C)|\leq 8$ to deduce that $\cM$ has rank at most ${{7}\choose{2}}=21$.
\end{proof}

We can now complete the proof by observing that the matroid $\cM_{g_{1,1,-1}}(K_n)$ from  Section \ref{sec:induced} is a partial elevation of $\cU_\cX$ and has rank $n+1$.
Since $\cM$ is a maximal partial elevation of $\cU_\cX$  by Lemma \ref{lem:free_elevation} and has rank at most $36$ by Claim~\ref{claim:K4K232},  $\cM$ is not the unique  maximal partial elevation of $\cU_\cX$ for all $n\geq 36$.
This completes the proof.
\end{proof}

The above  proof also implies that there are two distinct maximal $\{K_4, K_{2,3}\}$-matroids on $K_n$ for all $n\geq 36$. The only modification is that we use Lemma \ref{lem:free_elevation1} in the final paragraph to deduce that $\cM$ is a maximal $\{K_4, K_{2,3}\}$-matroid on $K_n$.


\begin{figure}[t]
\centering
\begin{minipage}{0.45\textwidth}
\centering
\includegraphics[scale=0.6]{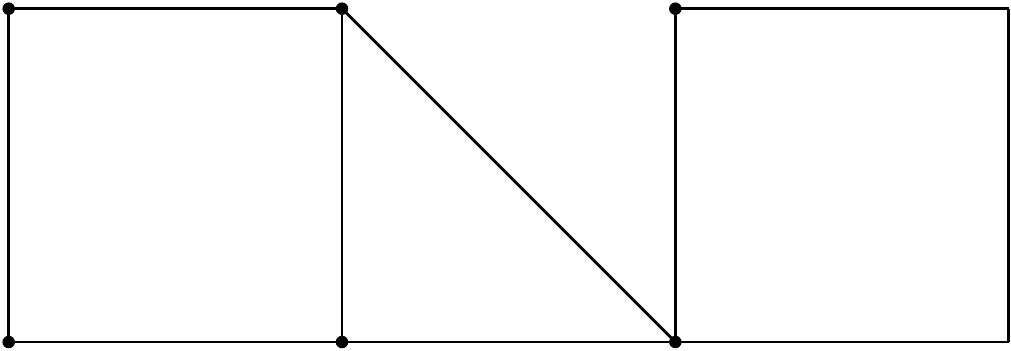}
\par
$G$
\end{minipage}
\caption{The graph $G$ in the proof of Theorem~\ref{thm:K4K23}.}
\label{fig:K4K23v2}
\end{figure}

\section{Matroids from Rigidity, Hyperconnectivity and Matrix Completion}
\subsection{\boldmath Rigidity matroids, cofactor matroids and $K_{d+2}$-matroids}
We are given a generic realisation $p:V(K_n)\to \R^d$ and we would like to know when a subgraph $G\subset K_n$ is {\em $d$-rigid} i.e., every continuous motion of the vertices of $(G,p)$ which preserves the distances between adjacent pairs of vertices must preserve the distances between all pairs of vertices.  The (edge sets of the) minimal $d$-rigid spanning subgraphs of $K_n$ are the bases of a matroid $\cR_d(K_n)$ which is referred to as the {\em $d$-dimensional   generic rigidity matroid}. It is well known that $\cR_d(K_n)$ is a  $K_{d+2}$-matroid on $K_n$ and
$\cR_1(K_n)$ is the cycle matroid of $K_n$. Pollaczek-Geiringer \cite{P-G} and subsequently Laman~\cite{Lam} showed that $\cR_2(K_n)=\cM_{f_{2,3}}(K_n)$. Charactising $\cR_d(K_n)$ for $d\geq 3$ is an important open problem in discrete geometry.

Graver \cite{G} suggested we may get a better understanding of $\cR_d(K_n)$ by studying the poset of all {\em abstract $d$-rigidity matroids on $K_n$}. This can be defined, using a result of Nguyen \cite{N}, as the poset of all $K_{d+2}$-matroids on $K_n$ of rank $dn-{{d+1}\choose{2}}$. Graver conjectured that $\cR_d(K_n)$ is the unique maximal element   in this poset and verified his conjecture for the cases when $d=1,2$. The same proofs yield the slightly stronger results given in Theorem~\ref{thm:K3}(a) and (b). 

Whiteley \cite{Wsurvey} showed that Graver's conjecture is false when $d\geq 4$ by showing that the cofactor matroid $\cC^{d-2}_{d-1}(K_n)$ from the theory of bivariate splines  is an abstract $d$-rigidity matroid for all $d\geq 1$, and that $\cR_d(K_n)\not\succ  \cC^{d-2}_{d-1}(K_n)$ for all $d\geq 4$ and sufficiently large $n$. He offered the revised conjecture
that $\cC^{d-2}_{d-1}(K_n)$ is the unique maximal element   in the poset
of all abstract $d$-rigidity matroids on $K_n$. We recently verified the case  $d=3$ of this conjecture in joint work with Clinch \cite{CJT1}. 

\begin{theorem}[\cite{CJT1}]\label{thm:cof}
The cofactor  matroid $\cC^{1}_{2}(K_n)$  is the unique maximal $K_5$-matroid on $K_n$ and $\val_{K_5}$ is its rank function.
\end{theorem}

We propose the following  strengthening of Whiteley's conjecture.

\begin{conjecture}\label{con:cof}
The cofactor  matroid $\cC^{d-2}_{d-1}(K_n)$  is the unique maximal $K_{d+2}$-matroid on $K_n$ for all $d\geq 1$ and $\val_{K_{d+2}}$ is its rank function.
\end{conjecture}

\subsection{\boldmath Birigidity and rooted $K_{s,t}$-matroids on $K_{m,n}$}

Let $H$ be a bipartite graph with bipartition $(A,B)$ and $K_{m,n}$ be a copy of the complete bipartite graph with bipartition $(U,W)$ where $|U|=m$ and $|W|=n$. 
We say that a subgraph $H'$ of $K_{m,n}$ is a {\em rooted copy of} $H$ in $K_{m,n}$ if there is an isomorphism $\theta$ from $H$ to $H'$ with $\theta(A)\subseteq U$ and 
$\theta(B)\subseteq W$. 
Let $\{H\}^*_{K_{m,n}}$ be the set of all rooted-copies of $H$ in $K_{m,n}$.
A matroid $\cM$ on $K_{m,n}$ is said to be a {\em rooted $H$-matroid} if 
it is a  $\{H\}^*_{K_{m,n}}$-matroid.
Note that the given ordered bipartition  $(A,B)$ of $H$ plays a significant role in this definition - we do not require that an isomorphic image $\theta(H)$  of $H$ in $K_{m,n}$ is a circuit in $\cM$ when $\theta(A)\not\subseteq U$. On the other hand, if $H$ has an automorphism which maps $A$ onto $B$, then we will get the same matroid for each ordering of the bipartition of $H$ and this matroid will be equal to the (unrooted) $H$-matroid on $K_{m,n}$. 
%

\subsubsection{Birigidity matroids}
As a primary example of matroids on complete bipartite graphs, we shall introduce the birigidity matroids of Kalai, Nevo, and Novik \cite{KNN}.

Let $G=(U\cup W, E)$ be a bipartite graph with $m=|U|$ and $n=|W|$, 
$p:U\to \R^k$, and $q: W\to \R^\ell$. We assume that  the vertices of $U$ and $W$ are ordered as $u_1,u_2,\ldots,u_m$ and $v_1,v_2,\ldots,v_n$, respectively.
We define the {\em $(k,\ell)$-rigidity  matrix} of $(G,p,q)$, denoted by $R^{k,\ell}(G,p,q)$, to be the matrix of size 
$|E|\times (\ell m+kn)$ in which 
each vertex in $U$ labels a set of $\ell$ consecutive columns from the first $\ell m$ columns, each vertex in $W$ labels a set of $k$ consecutive columns from the last $kn$ columns, 
each row is associated with an edge,
and the row labelled by the  edge $e=u_iw_j$ 
is 
\[
\kbordermatrix{
 & & u_i & & w_j & \\
 e=u_iw_j & 0\dots 0 & q(w_j) & 0 \dots 0 & p(u_i) & 0\dots 0 
}.
\]
 The generic $(k,\ell)$-rigidity matroid $\cR_{m,n}^{k,\ell}$ is the row matroid 
 of $R^{k,\ell}(K_{m,n},p,q)$ for any generic $p$ and $q$.
It can be checked that the rank of   $\cR^{k,\ell}_{m,n}$ is equal to $\ell m+kn-k\ell$, from which it follows that $K_{k+1,\ell+1}$ is a circuit and $\cR_{m,n}^{k,\ell}$ is a rooted $K_{k+1,\ell+1}$-matroid.

As pointed in~\cite{KNN}, $\cR^{k,\ell}_{m,n}$  coincides with the picture lifting matroids extensively studied by Whiteley~\cite{Wlift}  when $\min\{k,\ell\}=1$. 
We will show that this matroid is the unique maximal rooted $K_{k+1,\ell+1}$-matroid in this case. 

\begin{theorem}\label{thm:K2k} 
$\cR_{m,n}^{k,1}$ is the unique maximal rooted $K_{k+1,2}$-matroid on $K_{m,n}$.
\end{theorem}
\begin{proof}
Whiteley~\cite{Wlift} showed that the picture lifting matroid is the matroid induced by 
the submodular, non-decreasing function $h:2^{E(K_{m,n})}\to\zed$ defined by
\[
h(F):=|U(F)|+k|W(F)|-k \quad (F\subseteq E(K_{m,n})),
\]
where $U(F)$ and $W(F)$ denote the sets of vertices in $U$ and $W$, respectively, that are incident to $F$.
Since every connected flat in $\cM_{h}(K_{m,n})$ is a complete bipartite graph $K_{m',n'}$ for some $m'\geq 1$ and $n'\geq 2$,
we may deduce the theorem from Lemma~\ref{lem:conflat} by  showing that 
$K_{m',n'}$ can be constructed by a weakly saturated, rooted $K_{k+1,2}$-sequence from a subgraph $G$ with $m'+kn'-k$ edges.
Such a sequence is easily obtained by taking 
$$E(G)=\{u_iw_1:1\leq i\leq m'\}\cup \{u_iw_j: 1\leq i\leq k \mbox{ and }2\leq j\leq n'\}.$$
\end{proof}
We refer the reader to~\cite{A} for more details on weakly saturated, rooted $K_{s,t}$-sequences in $K_{m,n}$.

Lemma \ref{lem:conflat} also tells us that the rank function of $\cR_{m,n}^{k,1}$ is determined by proper, rooted $K_{k+1,2}$-sequences. We conjecture that this extends to  $\cR^{k,\ell}_{m,n}$ for all $k,\ell\geq 1$.


\begin{conjecture}\label{conj:birid}
$\cR_{m,n}^{k,\ell}$ is the unique maximal rooted $K_{k+1,\ell+1}$-matroid on $K_{m,n}$ and the rank of any $F\subseteq E(K_{n,m})$ is given by 
\begin{equation*}
r(F)=\mbox{$\min \{{\rm val}(F,\cS): \text{$\cS$ is a proper, rooted $K_{k+1,\ell+1}$-sequence in $K_{m,n}$}\}$}. 
\end{equation*}
\end{conjecture}

The special case of this conjecture for $\cR_{m,n}^{2,2}$ is equivalent to a conjecture on the rank function of  $\cR_{m,m}^{2,2}$ given in \cite[Section 8]{JJT}.
Bernstein~\cite{Ber} gave an NP-type combinatorial characterization for independence in $\cR_{m,n}^{2,2}$,
but no co-NP-type characterization is known. The special case $k=\ell=2$ of Conjecture \ref{conj:birid} would provide such a certificate but even this special case seems challenging. As some evidence in support of the conjecture, we can show that Conjecture~\ref{con:unique} holds for the poset of $K_{3,3}$-matroids on $K_{m,n}$.
%
%
\begin{theorem}\label{thm:K33}
The following statements are equivalent.
\begin{enumerate}
\item[(a)] There is a unique maximal $K_{3,3}$-matroid on $K_{m,n}$.
\item[(b)] ${\rm val}_{K_{3,3}}$ is submodular on $K_{m,n}$.
\end{enumerate}
\end{theorem}
We will sketch a proof of Theorem~\ref{thm:K33} after Theorem~\ref{thm:K4K33} below (which gives an analogous result for $\{K_4,K_{3,3}\}$-matroids on $K_n$).

\subsection{\boldmath Hyperconnectivity matroids, matrix completion and\\ $\{K_d, K_{s,t}\}$-matroids on $K_n$}

Let  $p: V(K_n) \to \R^d$ be a generic map.
We assume that  
the vertices of $K_n$ are ordered as $v_1,v_2,\ldots,v_n$.
Kalai~\cite{Khyp} defined  the {\em $d$-hyperconnectivity matroid}, $\cH_n^d$, to be 
the row matroid of the matrix of size  ${n\choose2}\times dn$
in which 
each vertex of $K_n$ labels a set of $d$ consecutive columns,
each row is labelled by an edge of $K_n$,
and the row labelled by the  edge $e=v_iv_j$ with $i<j$ is 
\begin{equation}\label{eq:hyp}
\kbordermatrix{
 & & v_i & & v_j & \\
 e=v_iv_j & 0\dots 0 & p(v_j) & 0 \dots 0 & -p(v_i) & 0\dots 0 
}.
\end{equation}
He showed that, when $n\geq 2d+2$, this matroid is a $\{K_{d+2},K_{d+1,d+1}\}$-matroid of rank $dn-{d+1\choose 2}$.  

As a variant of $\cH_n^d$, Kalai~\cite{Khyp} also introduced the matroid $\cI_n^d$, which is the row matroid of the 
$({n\choose2}\times dn)$-matrix with rows
\[
\kbordermatrix{
 & & v_i & & v_j & \\
 e=v_iv_j & 0\dots 0 & p(v_j) & 0 \dots 0 & p(v_i) & 0\dots 0 
}
\]
instead of (\ref{eq:hyp}).
He showed that, when $n\geq 2d+2$ and $d\geq 2$, $\cI_n^d$ is a $K_{d+1,d+1}$-matroid  on $K_n$ of rank $dn-{d\choose 2}$. 
In the special case when $d=2$, this rank constraint implies that $\cI_n^2$ is a $\{K_5,K_{3,3}\}$-matroid.

The matroids $\cH_n^d$ and $\cI_n^d$ arise naturally in the context of the rank $d$  completion problem for partially filled  $n\times n$ matrices which are skew-symmetric and symmetric, respectively, see~\cite{B,SC}. The restriction of either $\cI_n^d$ or $\cH_n^d$ to the complete bipartite graph $K_{m,n}$ is the birigidity matroid $\cR^{d,d}_{m,n}$, and this matroid 
arises in the context of the rank $d$ completion problem for partially filled  $m\times n$ matrices, see~\cite{SC}.

When $d=1$, $\cH_n^1$ is the cycle matroid (and hence is the unique maximal $\{K_{3},K_{2,2}\}$-matroid on $K_n$ by Theorem \ref{thm:K3}(a)) and   $\cI_n^1$  is the even cycle matroid  (and hence is the unique maximal $K_{2,2}$-matroid on $K_n$ by Theorem \ref{thm:countbip}(a)). 


We can find one more example of a $\{K_d, K_{s,t}\}$-matroid in rigidity theory.  Bolker and Roth \cite{BR} showed that $K_{d+2,d+2}$ is a circuit in the $d$-dimensional rigidity matroid $\cR_d(K_n)$ when $d\geq 3$.
Hence   $\cR_d(K_n)$ is a $\{K_{d+2},K_{d+2,d+2}\}$-matroid on $K_n$ for all $d\geq 3$.

We conjecture that each of $\cH^d_n$, $\cI^2_n$ and $\cR^d_n$ is  the unique maximal matroid in its respective poset.

\begin{conjecture}\label{conj:multi}
(a) For $n\geq 2d+2$, $\cH_n^{d}$ is the unique maximal  $\{K_{d+2},K_{d+1,d+1}\}$-matroid on $K_{n}$ and its rank function is $\val_{\{K_{d+2},K_{d+1,d+1}\}}$. \\
(b) For  $d=2$ and $n\geq 6$, $\cI_n^{2}$ is the unique maximal  $\{K_5,K_{3,3}\}$-matroid on $K_{n}$ and its rank function is $\val_{\{K_5,K_{3,3}\}}$. \\
(c) For  $d\geq 3$ and $n\geq 2d+4$, $\cR_n^{d}$ is the unique maximal  $\{K_{d+2},K_{d+2,d+2}\}$-matroid on $K_{n}$ and its rank function is $\val_{\{K_{d+2},K_{d+2,d+2}\}}$. 
\end{conjecture}

We close this section by considering the special case of Conjecture \ref{conj:multi}(a) when $d=2$.

\subsubsection{\boldmath $\{K_4, K_{3,3}\}$-matroids on $K_n$}
Understanding the poset of all $\{K_4, K_{3,3}\}$-matroids on $K_n$ is important since 
these matroids appear in applications such as the rank two completion of partially filled skew-symmetric matrices and partially-filled rectangular matrices, see~\cite{B,SC}.
We shall prove that, if this poset has a unique maximal element, then the rank function of the maximal element is $\val_{\{K_4, K_{3,3}\}}$.
This confirms Conjecture~\ref{con:unique} for $\{K_4, K_{3,3}\}$-matroids. We will need two general results for a matroid on the edge set of a graph. The first was   
proved for the special case of abstract rigidity matroids in \cite{CJT2}. The same proof gives: 

\begin{lemma}\label{lem:deg}
Let $\cM$ be a matroid defined on the edge set of a graph $G$. 
Suppose that $G[C]$ is 2-connected for every circuit $C$ in $\cM$.
Then, for every connected set $X$ in $\cM$,   
\[
\sum_{v\in V(X)} \min\{d_B(v): B \text{ is a basis of } X\}\leq 2(r_{\cM}(X)-1)-|V(X)|.
\]
\end{lemma}

Our second lemma, concerns a well known graph operation from rigidity theory.
Given  a graph $G$, the {\em 0-extension operation} constructs a new graph by adding a new vertex $v_0$ and two edges $v_0v_1$ and $v_0v_2$ with distinct $v_1,  v_2\in V(G)$.
We say that a  
matroid $\cM$ on $K_n$ has the {\em 0-extension property} if every 0-extension preserves  independence in $\cM$, i.e. $E(G')$ is independent if $E(G)$ is independent and $G'$ is obtained from $G$ by a 0-extension operation for all $G, G'\subseteq K_n$. 
\begin{lemma}\label{lem:con}
Let $\cM$ be a $K_4$-matroid  on $K_n$ with the 0-extension property.
Then, every circuit in $\cM$ induces a  2-connected subgraph of $K_n$.
\end{lemma}
\begin{proof}
Suppose, for a contradiction,  that some circuit $C$ in $\cM$ does not induce a 2-connected subgraph of $K_n$.

We first consider the case when $C$ is connected.
Then $C$ can be partitioned into two sets $X$ and $Y$ such that 
$|V(X)\cap V(Y)|=1$.
Let $K$ be the edge set of the complete graph on $V(Y)$.
Since $\cM$ is a $K_4$-matroid, Theorem \ref{thm:K3}(b) gives $r_\cM(K)\leq 2|V(Y)|-3$.
The fact that $X\cup Y$ is a circuit now gives
$r_\cM(X\cup K)\leq r_\cM(X)+r_\cM(K)-1\leq |X|+2|V(Y)|-4$.

We may construct an independent subset of $X\cup K$ by extending the independent set $X$ using 0-extensions.
Let $e$ be an edge in $K$ incident to the vertex in $V(X)\cap V(Y)$.
Then $X+e$ is independent by the 0-extension property.
Repeatedly applying the 0-extension operation, we can extend $X+e$ to an independent set $B$ 
of size $|X|+1+2(|V(Y)|-2)=|X|+2|V(Y)|-3$ by adding  edges in $K$.
This contradicts the fact that the rank of $X\cup K$ is at most $|X|+2|V(Y)|-4$.

The case when $C$ is not connected can be proved similarly.
\end{proof}

We need one more graph operation.
Given a vertex $v_1$ of a graph $G$, 
the {\em diamond splitting operation} 
at $v_1$ (with respect to a fixed partition $\{U_0, U^*, U_1\}$ of $N_G(v_1)$ with $|U^*|=2$) removes the edges between $v_1$ and the vertices in $U_0$, adds a new vertex $v_0$, and adds new edges $v_0u$ for all $u\in U_0\cup U^*$. 
We say that a matroid $\cM$ on $K_n$ has the {\em diamond splitting property} if any diamond splitting operation preserves independence in $\cM$.
It was shown in~\cite{KNN} that $\cH_n^2$ has both the 0-extension property and the diamond splitting property.
 
We can now prove our main result on $\{K_4, K_{3,3}\}$-matroids.

\begin{theorem}\label{thm:K4K33}
Let $\cX=\{K_4, K_{3,3}\}_{K_n}$. 
Then the following statements are equivalent.
\begin{enumerate}
\item[(a)] There is a unique maximal $\cX$-matroid on $K_n$.
\item[(b)] The free elevation of $\cU_{\cX}$ has the 0-extension property and the diamond splitting property.
\item[(c)] There is an $\cX$-matroid  on $K_n$ that has the 0-extension property, the diamond splitting property, and the $\cX$-covering property.
\item[(d)] ${\rm val}_{\cX}$ is submodular on $E(K_n)$.
\end{enumerate}
\end{theorem}
\begin{proof}
(d) $\Rightarrow$ (a): This follows from Lemma~\ref{lem:upper1}.
\\[1mm]
(a) $\Rightarrow$ (b): Since $\cX$ is 6-uniform and union-stable, $\cU_{\cX}(K_n)$ is a maximal matroid in the poset of all $\cX$-matroids on $K_n$ of rank at most 6. Clearly $\cU_{\cX}(K_n)\neq \cU_5(K_n)$. Lemma~\ref{lem:free_elevation0} and (a) now imply that the free elevation of $\cU_{\cX}(K_n)$ is the unique maximal $\cX$-matroid on $K_n$.
Since $\cH_n^2$ is an $\cX$-matroid on $K_n$ with the 0-extension property and the diamond splitting property,  the free elevation of $\cU_{\cX}(K_n)$ also has the 0-extension property and the diamond splitting property.
\\[1mm]
(b) $\Rightarrow$ (c): This follows from Lemma~\ref{lem:cover}.
\\[1mm]
(c) $\Rightarrow$ (d): 
Suppose that (c) holds for some $\cX$-matroid $\cM$ on $K_n$.
We prove that $r_{\cM}={\rm val}_{\cX}$.
By Lemma~\ref{lem:conflat}, it suffices to show that, for each  connected flat $F$ of $\cM$,
there is a proper $\cX$-sequence $\cS$ such that $r_{\cM}(F)={\rm val}(F,\cS)$.
We prove this by induction on  the rank of $F$.

%
%
%

Since $\cM$ has the 0-extension property, Lemma~\ref{lem:con} 
implies that every circuit in $\cM$ induces a 2-connected subgraph of $K_n$.
Since $\cM$ is a $K_4$-matroid, $r_\cM(F)\leq 2|V(F)|-3$ and Lemma~\ref{lem:deg} now implies that there exists a base $B$ of $F$ and a vertex $v\in V(B)$ such that 
$d_B(v)\leq 2$. Let $F_v$ and $B_v$ be the set of edges in $F$ and $B$, respectively, which are not incident to $v$. We first show that
\begin{equation}\label{eq:flat}
\mbox{$F_v$ is a flat in $\cM$, $d_B(v)=2$ and $r_\cM(F_v)=r_\cM(F)-2$.}
\end{equation}
To verify (\ref{eq:flat}) we first note that, since $\cM$ has the 0-extension property, every circuit in $\cM$ has minimum degree at least three. Since $d_B(v)\leq 2$, this implies that $\cl_\cM(B_v)=F_v$ and hence $F_v$ is a flat in $\cM$. 
In addition, since $F$ is connected in $\cM$, we have $d_F(v)\geq 3$. The facts that $d_B(v)\leq 2$ and $\cM$ has the 0-extension property, now give 
$|B|=r_\cM(F)\geq r_\cM(F_v)+2=|B_v|+2\geq |B|$.  Hence equality holds throughout and 
(\ref{eq:flat}) holds.

\begin{claim}\label{clm:key}
Let $x, y\in N_F(v)$ and $z\in V(F_v)$ be three distinct vertices.
Suppose that $xz,yz\in F_v$. Then $uz\in F_v$  for all $u\in N_F(v)\setminus \{z\}$.
\end{claim}
\begin{proof}
Suppose, for a contradiction, that $uz\not\in F_v$ for some $u\in N_F(v)\setminus \{z\}$.
Since $F_v$ is a flat, $B_v+uz$ is independent.
We may construct  $B'=B_v\cup \{vx,vy,vu\}$ from $B_v+uv$ by applying the diamond splitting operation to  $z$ in such a way that 
the new vertex $v$ has degree three and is adjacent to $x, y, u$.
Then $B'$ is contained in $F$ and is independent in $\cM$.
Since $|B'|>|B|$, this contradicts the fact that $B$ is a base of $F$.
\end{proof}

We may apply induction to each connected component of $F_v$ in $\cM$ to obtain a proper $\cX$-sequence $\cS'=(X_1,X_2\dots, X_t)$ such that $r_{\cM}(F_v)={\rm val}(F_v,\cS')$.
Since $\cM$ has the  $\cX$-covering property, $F$ is the union of copies of $K_4$ and $K_{3,3}$.
Let $N_F(v)=\{u_1, u_2,\dots, u_k\}$.

Suppose that some edge of $F$ which is incident to $v$ is contained in a copy $K_4$ in $F$. Relabelling if necessary, we may suppose that the complete graph $K(v,u_1,u_2,u_3)$ satisfies $K(v,u_1,u_2,u_3)\subseteq F$. We will show that  $K(v,u_1,u_2,\ldots,u_t)\subseteq F$. For each $i=4,5,\ldots,t$, we may apply Claim \ref{clm:key} with $x=u_1,y=u_2,z=u_3,u=u_i$ to deduce that $u_iu_1,u_iu_2\in F$. Since $\cM$ has the  0-extension property, this implies that 
$F$ contains an independent set of size $2|N_F(v)|-3$ on $N_F(v)$. Since $F$ is a flat and every $A\subseteq E(K_n)$ has rank at most $2|V(A)|-3$,
this implies that $K(v,u_1,u_2,\ldots,u_t)\subseteq F$.
Let $X_{t+i}$ be a copy of $K_4$ on $\{v, u_i, u_{i+1}, u_{i+2}\}$ for $i=1,\dots, k-2$,
and let $\cS=(X_1,\dots, X_t, X_{t+1},\dots, X_{t+i+2})$ be obtained by appending $(X_{t+1},\dots, X_{t+k-2})$ to $\cS'$.
Then we have 
${\rm val}(F,\cS)={\rm val}(F_v, \cS')+2=r_{\cM}(F_v)+2= r_{\cM}(F)$, 
as required.

It remains to consider the case when no edge of $F$ incident to $v$ is contained in a copy  of $K_4$ in $F$.
Then every edge in $F$ which is incident to $v$ is contained in  a copy of $K_{3,3}$.
Relabelling if necessary we may suppose that the complete bipartite graph $K(v,w_1,w_2;u_1,u_2,u_3)$ is contained in  $F$.
Then $w_i\notin N_F(v)$ for $i=1,2$, since otherwise the facts that 
$F$ is a flat and $\cM$ is a $K_4$-matroid would imply that $K(v, u_{1}, u_{2}, w_i)\subseteq F$. 
By Claim~\ref{clm:key}, $F$ contains $u_iw_1$ and $u_iw_2$ for all $u_i\in N_F(v)$.
Hence, $F$ contains the complete bipartite graph $K(N_F(v);\{w_1, w_2\})$.
Let $X_{t+i}=K(v, w_1, w_2;u_i, u_{i+1}, u_{i+2})$ for $i=1,\dots, k-2$,
and let $\cS=(X_1,\dots, X_t, X_{t+1},\dots, X_{t+i+2})$ be obtained by appending $(X_{t+1},\dots, X_{t+k-2})$ to $\cS'$.
Then we have ${\rm val}(F,\cS)={\rm val}(F_v,\cS')+2=r_{\cM}(F_v)+2= r_{\cM}(F)$, 
as required.

This completes the proof of the theorem.
\end{proof}

We can prove Theorem~\ref{thm:K33}
by restricting the above argument to complete bipartite graphs.
We need the following counterpart to Lemma~\ref{lem:con}.
\begin{lemma}\label{lem:con2}
Let $\cM$ be a $K_{3,3}$-matroid on $K_{m,n}$ with the 0-extension property.
Then, every circuit in $\cM$ induces a  2-connected subgraph of $K_{m,n}$.
\end{lemma}
\begin{proof}
Suppose, for a contradiction, that some circuit $C$ in $\cM$ does not induce a 2-connected subgraph of $K_{m,n}$.

We first consider the case when $K_{m,n}[C]$ is connected.
Then $C$ can be partitioned into two sets $X$ and $Y$ such that 
$|V(X)\cap V(Y)|=1$.

Suppose that $|U(Y)|=1$ or $|W(Y)|=1$.
Then $Y$ is a tree.
Since $X$ is independent, the 0-extension property implies that $X\cup Y$ is independent, which is a contradiction. Hence, $|U(Y)|\geq 2$ and  $|W(Y)|\geq 2$.

Let $K=K(U(Y);W(Y))$.
Since $\cM$ is a $K_{3,3}$-matroid on $K_{m,n}$, $r_\cM(E(K))\leq 2|V(Y)|-4$.
The fact that $X\cup Y$ is a circuit now gives
$r_\cM(X\cup E(K))\leq r_\cM(X)+r_\cM(E(K))-1\leq |X|+2|V(Y)|-5$.

We may construct an independent subset of $X\cup E(K)$ by extending the independent set $X$ using 0-extensions.
Without loss of generality, we can assume $X$ and $Y$ share a vertex $u$ in $U$.
By $|U(Y)|\geq 2$, we can choose $v\in U(Y)\setminus \{u\}$.
Let $G$ be the graph obtained from $K_{m,n}[X]$ by appending $v$ as an isolated vertex.
Then, repeatedly applying the 0-extension operation, we can extend $G$ to a graph $G'$ 
of size $|X|+2(|V(Y)|-2)=|X|+2|V(Y)|-4$.
This contradicts the fact that the rank of $X\cup E(K)$ is at most $|X|+2|V(Y)|-5$.

The case when $K_{m,n}[C]$ is disconnected is similar.
\end{proof}

\begin{proof}[Proof of Theorem~\ref{thm:K33}]
The proof proceeds in the same way as that of Theorem~\ref{thm:K4K33} by using Lemma~\ref{lem:con2} instead of Lemma~\ref{lem:con}.
\end{proof}

\paragraph{Acknowledgements} We would like to thank Gil Kalai for bringing our attention to weakly saturated sequences.

This work was supported by 
JST ERATO Grant Number JPMJER1903,  
JSPS KAKENHI Grant Number 18K11155, and EPSRC overseas travel grant EP/T030461/1.


\begin{thebibliography}{99}

\bibitem{A}
N.~Alon, An extremal problem for sets with applications to graph theory, Journal of Combinatorial Theory, Series A. 40, (1985), 82--89

\bibitem{Ber} D.~I.~Bernstein, Completion of tree metrics and rank 2 matrices, Linear Algebra
Appl. 533, (2017), 1-13.

\bibitem{BR}
E.~D.~Bolker and B.~Roth, When is a bipartite graph a rigid framework?. Pacific J.~Math, 90, (1980), 27--44.

\bibitem{B}
B.~Bollob\'as,  Weakly $k$-saturated graphs, in Proc. Coll. Graph Theory, Ilmenau  (1967) 25--31.


\bibitem{B86}
T.~Brylawski, Constructions, Theory of Matroids, N.L.~White, ed., Cambridge Univ.~Press, Cambridge (1986), 127--223.


\bibitem{CJT1}
K.~Clinch, B.~Jackson and S.~Tanigawa,
Abstract 3-rigidity and bivariate $C_2^1$-splines I: Whiteley's
maximality conjecture, preprint available at \url{https://arxiv.org/abs/1911.00205}.

\bibitem{CJT2}
K.~Clinch, B.~Jackson and S.~Tanigawa,
Abstract 3-rigidity and bivariate $C_2^1$-splines II: Combinatorial Characterization, preprint available at \url{https://arxiv.org/abs/1911.00207}.

\bibitem{C} 
H.~Crapo, Erecting geometries, Proc.~2nd Chapel Hill Conf.~on Comb.~Math, 1970, 74--99.









\bibitem{D87} R.~Duke, Matroid erection and duality, European Journal of Combinatorics, 8, 1987, 367--370.



\bibitem{E} J. Edmonds, Submodular functions, matroids, and certain polyhedra, in {\rm Combinatorial Structures and their Applications},
eds. R. Guy, H. Hanani, N. Sauer, and J. Sch\"{o}nheim, Gordon and Breach, New York, 1970, 69-87.


\bibitem{G} J.~E.~Graver, Rigidity matroids, SIAM Journal on Discrete Mathematics, 4, 1991,  355--368.

\bibitem{GSS} J. E. Graver, B. Servatius, and H. Servatius, Combinatorial rigidity,  Amer.~Math.~Soc., 1993.


\bibitem{ike} R.~Ikeshita and S.~Tanigawa,
Count matroids of group-labeled graphs, Combinatorica, 38:1101--1127, 2018.










\bibitem{JJT}
B.~Jackson, T.~Jord{\'a}n and S.~Tanigawa,
Combinatorial conditions for the unique completability of low rank matrices,
{SIAM J. Discrete Math.} 28 (2014)  1797--1819.

\bibitem{Khyp}
G.~Kalai,  Hyperconnectivity of graphs. Graphs Combin. 1  (1985) 65--79.

\bibitem{Ksym}
G.~Kalai,
 Symmetric matroids. J. Combin. Theory Ser. B 50 (1990) 54--64.



\bibitem{KNN}
G.~Kalai, E.~Nevo, and I.~Novik,
Bipartite rigidity,
Trans. Amer. Math. Soc. 368 (2016) 5515--5545.

\bibitem{Lam}
G.~Laman, On graphs and the rigidity of skeletal structures, J. Eng. Math., 
4  (1970) 331--340.

%

%

\bibitem{MS}
G.~Moshkovitz and A.~Shapira,
Exact bounds for some hypergraph saturation problems,
Journal of Combinatorial Theory, Series B, 111, 2015, 242--248.




\bibitem{N} V.~H.~Nguyen, On abstract rigidity matroids, 
SIAM Journal on Discrete Mathematics, 24,  2010, 363--369.



\bibitem{Pap}
G.~Pap, 	A note on maximum matroids of graphs, The EGRES Quick-Proofs series, QP-2020-02, 2020.


\bibitem{P1}
O.~Pikhurko,  Uniform families and count matroids,  Graphs Combin., 17 (2001) 729--740.

\bibitem{P2}
O.~Pikhurko,  
Weakly saturated hypergraphs
and exterior algebra, Comb., Prob.  Comp., 10 (2001)  435--451.



\bibitem{P-G} H. Pollaczek-Geiringer, \"{U}ber die Gliederung ebener Fachwerke, Zeitschrift f\"{u}r Angewandte Mathematik und Mechanik (ZAMM), 7,  1927, 58--72.



\bibitem{SC} A.~Singer and M.~Cucuringu, Uniqueness of low-rank matrix completion by rigidity theory. SIAM Journal on Matrix Analysis and Applications, 31, 4, 2010, 1621--164.

\bibitem{Stalk} M. Sitharam, Recent developments in 3D bar-joint rigidity characterization, presentation at the BIRS Workshop: Advances in Combinatorial and Geometric Rigidity Theory,  2015,  
available at \\
{\tt http://www.birs.ca/events/2015/5-day-workshops/15w5114/videos} 


\bibitem{SV} S.~Sitharam and A.~Vince, The maximum matroid of a graph, preprint available at {\tt https://arxiv.org/pdf/1910.05390.pdf}

\bibitem{SW} B.~Schulze and W.~Whiteley, Rigidity of symmetric frameworks, 
in Handbook of Discrete and Computational Geometry, Third Edition, Editors: Csaba D. Toth, Joseph O'Rourke, Jacob E. Goodman, Chapman and Hall/CRC, 2017.








%
%
%
%
\bibitem{Wlift} W.~Whiteley, A matroid on hypergraphs with applications in scene analysis and geometry.
Discrete \& Computational Geometry, 4(1):75--95, 1989

\bibitem{Wsurvey} {W. Whiteley},
Some matroids from discrete applied geometry, in Matroid theory (Seattle, WA, 1995), 171--311,
Contemp. Math., 197, Amer. Math. Soc., Providence, RI, 1996.


\end{thebibliography}
\end{document}